\documentclass[12pt,amstex]{article}  
\usepackage{tikzsymbols}

\definecolor{cof}{RGB}{219,144,71}
\definecolor{pur}{RGB}{186,146,162}
\definecolor{greeo}{RGB}{91,173,69}
\definecolor{greet}{RGB}{52,111,72}

\usepackage{color}              %
\usepackage{graphicx}
\usepackage{amsmath,amssymb}
\usepackage[textsize=small]{todonotes}
\usepackage{hyperref}
\usepackage{enumerate, amsfonts, amsthm, bbm}
\usepackage{tikz}
\usetikzlibrary{arrows,automata,positioning}

\definecolor{MyDarkBlue}{rgb}{0,0.08,0.50}
\definecolor{BrickRed}{rgb}{0.65,0.08,0}

\hypersetup{
colorlinks=true,       %
    linkcolor=MyDarkBlue,          %
    citecolor=BrickRed,        %
    filecolor=red,      %
    urlcolor=cyan           %
}

\theoremstyle{plain}
\newtheorem{LEM}{Lemma}

\newtheorem{THM}{Theorem}

\newtheorem{COR}{Corollary}

\newtheorem{REM}{Remark}
\newtheorem{EXA}{Example}
\newtheorem{OP}[EXA]{Open Problem}

\newcommand{\nc}[1]{}

\newcommand{\1}{\mathbbm{1}}
\newcommand{\indic}[1]{\1_{\{#1\}}}
\newcommand{\nn}{\nonumber}

\numberwithin{equation}{section}

\renewcommand{\P}{\mathbb{P}}
\newcommand{\E}{\mathbb{E}}
\newcommand{\N}{\mathbb{N}}
\renewcommand{\H}{\mathbb{H}}
\newcommand{\Q}{\mathbb{Q}}
\newcommand{\R}{\mathbb{R}}
\newcommand{\Z}{\mathbb{Z}}

\newcommand{\DD}{\mathbb{D}}

\newcommand{\floor}[1]{\lfloor #1 \rfloor}

\newcommand{\mc}[1]{\mathcal{#1}}

\newcommand{\blank}[1]{}
\newcommand{\sss}{\scriptscriptstyle}

\newcommand{\vep}{\varepsilon}

\newcommand{\tempend}{\end{document}}
\newcommand{\bs}[1]{{\boldsymbol #1}}
\newcommand{\bo}{\bs{\omega}}

\newcommand{\sg}{\mc{S}_p^{\textrm{good}}}

\newcommand{\orth}{\mathbb{O}_+}

\DeclareMathOperator{\sgn}{sgn}

\begin{document}

\author{
Nicholas Beaton\footnote{\tt nrbeaton@unimelb.edu.au}, Mark Holmes\footnote{\tt holmes.m@unimelb.edu.au}, and Xin Huang\footnote{\tt huangxinop@gmail.com}\\
\small{School of Mathematics and Statistics}\\
\small{The University of Melbourne}}

\title{Chemical distance 
for the	half-orthant model.
}

\maketitle

\begin{abstract}
The half-orthant model is a partially oriented model of a random medium involving a parameter $p\in [0,1]$, for which there is a critical value $p_c(d)$ (depending on the dimension $d$) below which every point is reachable from the origin.  
We prove a limit theorem for the graph-distance (or ``chemical distance") for this model when  $p<p_c(2)$, and also when $1-p$ is larger than the critical parameter for site percolation in $\Z^d$.    The proof involves an application of the subadditive ergodic theorem.  Novel arguments herein include the method of proving that the expected number of steps to reach any given point is finite, as well as an argument that is used to show that the shape is ``non-trivial'' in certain directions.

\end{abstract}



\section{Introduction and main results}
The half-orthant model is an i.i.d.~site-based model of a random environment defined as follows:    Independently at each site $x\in \Z^d$, with probability $p\in [0,1]$ insert arrows from $x$ pointing to each of the  neighbours $x+e_i$, $i\in [d]:=\{1,2,\dots, d\}$, and otherwise (with probability $1-p$) insert arrows from $x$ to all neighbours $x\pm e_i$ of $x$.   Here $(e_i)_{i=1}^d$ are the canonical basis vectors. Call the sites that get all outward arrows \emph{full sites} and the others \emph{half sites}.  Let $\Omega_+$ be the set of half sites.

For fixed $d\ge 2$, let $\bs{\omega}=(\omega_x)_{x \in \Z^d}$ be such an environment and let $\gamma=(\gamma_i)_{i=0}^\ell$ be a nearest-neighbour path of length $\ell<\infty$ in $\Z^d$.  For $j\in [\ell]$, the $j$-th step of the path $\gamma_j-\gamma_{j-1}$ is \emph{consistent with} $\bs{\omega}$ if $\omega_{\gamma_{j-1}}$ contains the arrow pointing in direction $\gamma_j-\gamma_{j-1}$.  The path $\gamma$ is \emph{consistent with} $\bs{\omega}$ if the $j$-th step is consistent with $\bs{\omega}$ for every $j\in [\ell]$.   Figure \ref{fig:HO1} shows an example of a finite piece of an environment for this model, together with a path of length 8 that is consistent with this environment.

\begin{figure}
	\begin{center}
	\includegraphics[scale=0.5]{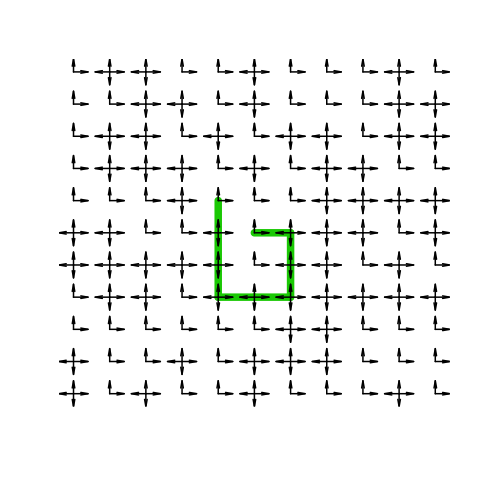}
	\end{center}
	\caption{An example of a finite piece of the environment for the half-orthant model, with the shortest path from the origin (centre) to $(-1,1)$ highlighted.}
	\label{fig:HO1}
\end{figure}

Consider the set of points $\mc{C}_o$ reachable from the origin $o=(0,0,\dots, 0)$ by following arrows.  This is exactly the set of points $x$ for which there exists a finite path from $o$ to $x$ that is consistent with the environment.   When $p=1$ this set is the positive orthant, i.e.~$\mc{C}_o(1)=\orth:=\{x\in \Z^d:x^{\sss [i]}\ge 0 \text{ for }i \in [d]\}$, where $x^{\sss [i]}$ denotes the $i$-th coordinate of $x$.  When $p=0$ this set is all of $\Z^d$, i.e.~$\mc{C}_o(0)=\Z^d$.  There is a natural coupling of environments for all values of $p$, such that the cluster $\mc{C}_o(p)$ is decreasing in $p$.  Moreover there exists a critical value $p_c(d)\in (0.57,1)$ such that almost surely $\mc{C}_o(p)=\Z^d$ if $p<p_c(d)$ and almost surely $\mc{C}_o(p)\ne \Z^d$ if $p>p_c(d)$.  See \cite{DREphase,DRE,DRE2} etc.

Understanding these kinds of site-based i.i.d.~environments lays the foundation for understanding random walks in i.i.d.~\emph{non-elliptic} (i.e.~some steps are not available from some sites) random environments see e.g.~\cite{RWDRE,RWDRE2}.  However in this paper we are concerned with a version of a shape theorem for this model.

A shape theorem for the forward cluster $\mc{C}_o$ as a set of points has been proved for this model when $p$ is large \cite{Shape}.  This result is similar to the kind of shape theorems available for oriented percolation and the contact process  (see e.g.~\cite{Dur84,DG}) - morally it says that the cluster looks like a cone from far away.  The fact that the half-orthant model is only partly oriented presents different  challenges compared to the oriented percolation setting.  For example, the result in \cite{Shape} currently does not extend to all $p>p_c(d)$ when $d\ge 3$ in part because of a lack of a ``sharpness'' result for this phase transition\footnote{Beekenkamp \cite{Beek},  has shown that the relevant exponential decay of certain connection probabilities holds for all $p$ larger than some $p_c'(d)$ which is believed (but not proved) to be equal to $p_c(d)$.}.

In this paper we consider the lengths of shortest paths (or \emph{chemical distance}) between vertices.  The chemical distance is a fundamental notion in the first passage percolation literature (see e.g.~\cite{ADH}).  A substantial difference is that in the half-orthant model edges are directed.  First passage percolation on directed edges has been considered recently in \cite{GM21} but our model is a site-based i.i.d.~environment (in  \cite{GM21} the environment is i.i.d.~over edges), so edges are locally  correlated in our setting.

For $u,v \in \Z^d$ and an environment $\bs{\omega}$ let $T_{u,v}=T_{u,v}(\bs{\omega})$ denote the length of any shortest path from $u$ to $v$ that is consistent with the environment.  If no such path exists then set $T_{u,v}=\infty$.   
For $v\in \Z^d$, let $T_v=T_{o,v}$.  In Figure \ref{fig:HO1}, $T_{(-1,1)}=T_{o,(-1,1)}=8$. When $p<p_c(d)$ we have that $\mc{C}_o=\Z^d$ almost surely, so $T_v$ is a.s.~finite.  By translation invariance, $T_{u,v}$ is almost surely finite for every $u,v\in \Z^d$.  Clearly, for any $v\in \Z^d$ we have $T_{v}\ge \|v\|_1:=\sum_{i=1}^d |v^{\sss [i]}|$.   In order to visualise our main result, it is useful to define $A_n=\{x\in \Z^d:T_{x}=n\}$, which is the random subset of sites $x$ for which the shortest path from the origin to $x$ has length $n$.   When $p=0$ this set is the $\ell_1$ ball $B_n:=\{x\in \Z^d:\|x\|_1=n\}$, and when $p=1$ this set is $B^+_n:=\{x\in \orth:\|x\|_1=n\}$.   Indeed, it is elementary that for every $p$, all points $v\in \Z^d$ in the positive orthant $\orth$ can be reached in $\|v\|_1$ steps.  The other directions are the interesting ones for this model. 
The first row in Figure \ref{fig:nick1} shows a simulation of the set $A_{4000}$ when $p=0.25,0.5,0.75$.
\begin{figure}
	\begin{center}
		\includegraphics[scale=0.04]{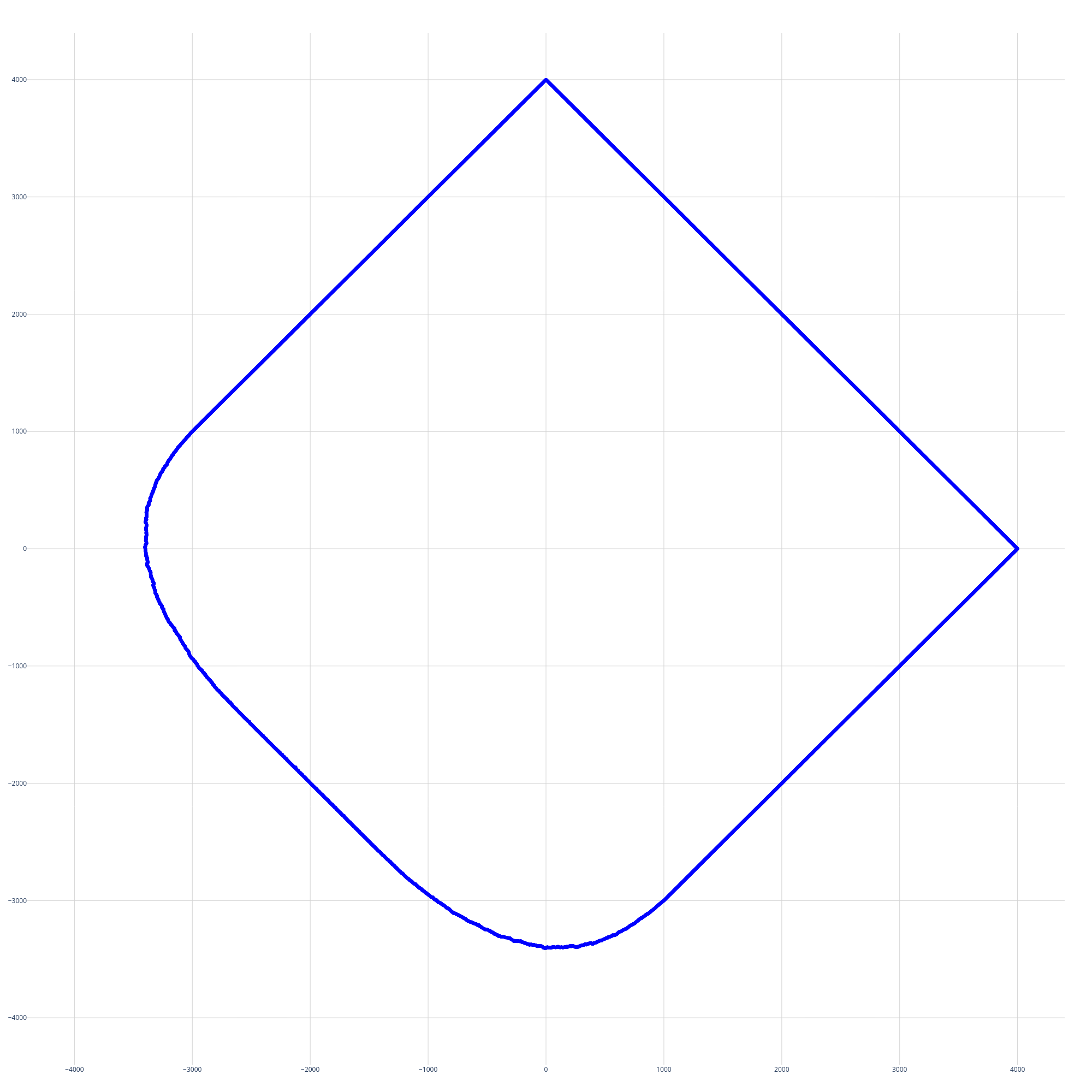}
		\hspace{1.5cm}
		\includegraphics[scale=0.04]{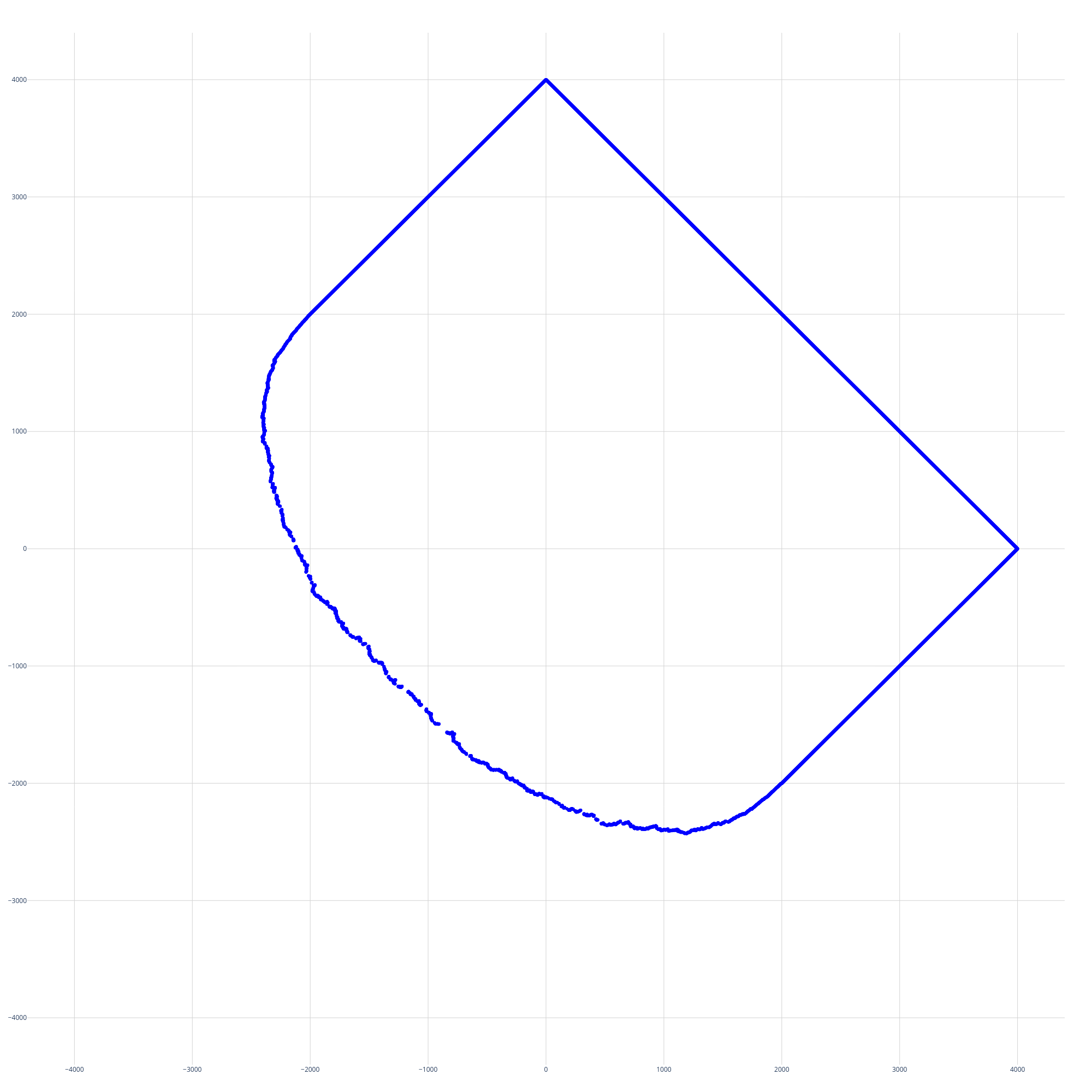}
		\hspace{1.5cm}
		\includegraphics[scale=0.04]{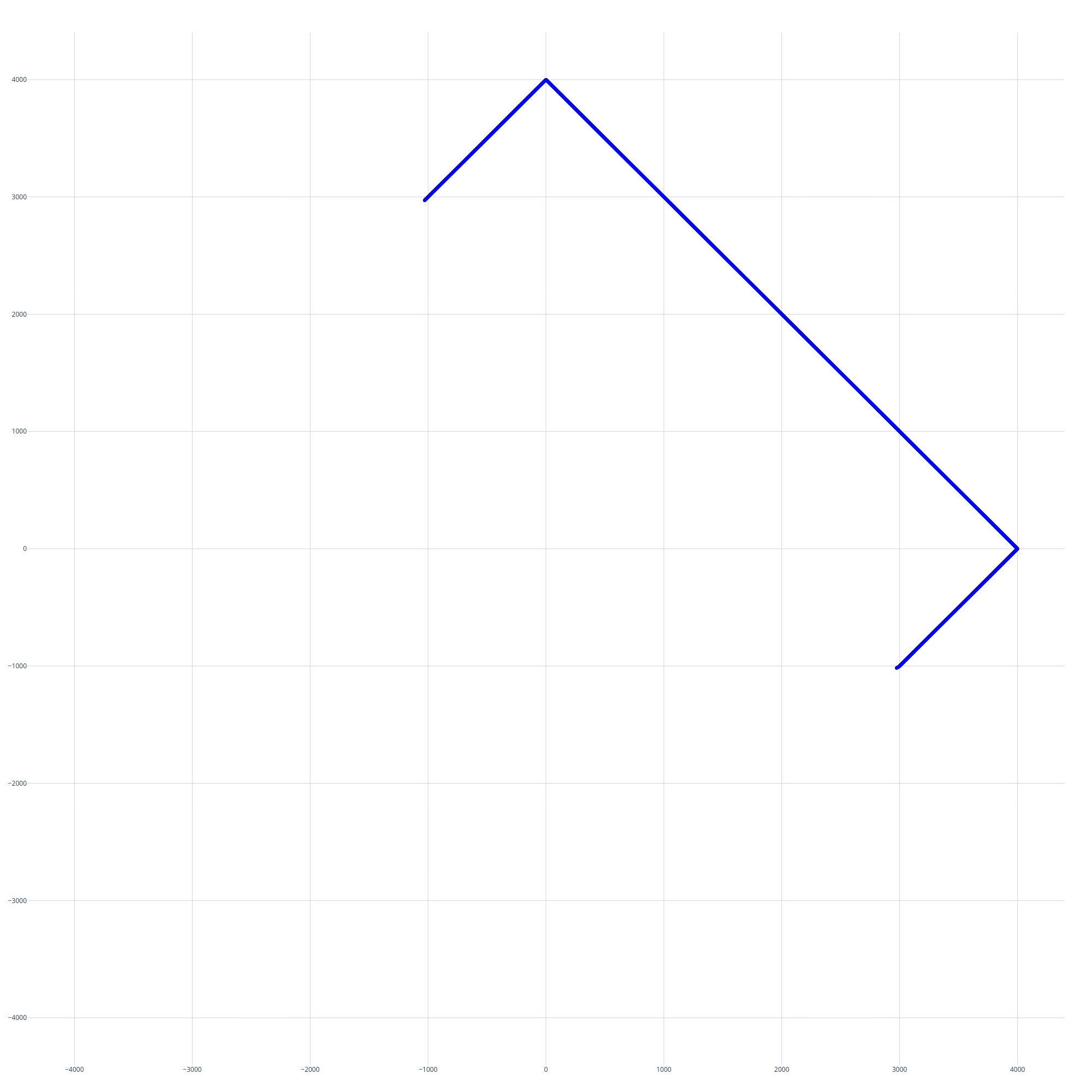}	\\
		\includegraphics[scale=0.25]{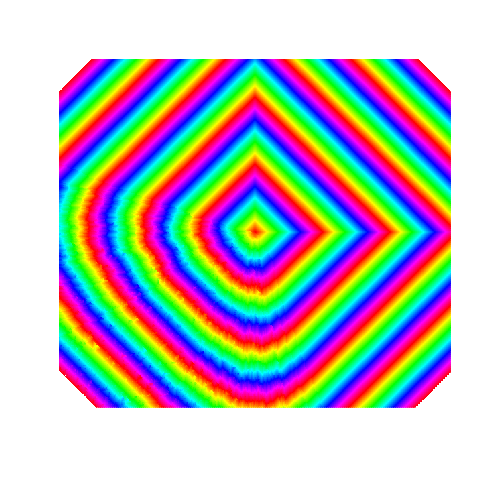}
		\includegraphics[scale=0.25]{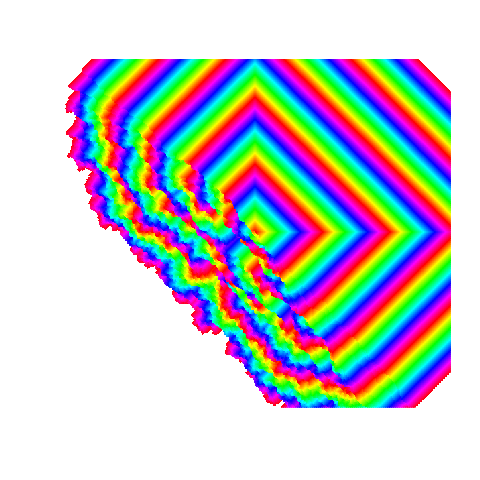}
		\includegraphics[scale=0.25]{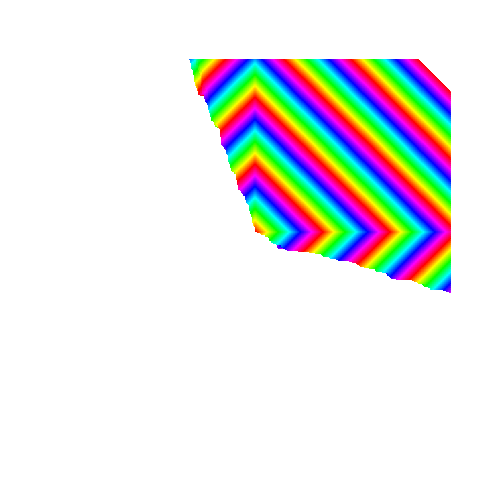}		
	\end{center}
	\caption{Top row: Simulation of the set of points whose distance from the origin within the random environment is exactly $n=4000$, for $p=0.25, p=0.5,p=0.75$ (left to right) respectively.  Bottom row: Simulations of these set for different values of $n$, for $p=0.25, p=0.5,p=0.75$.  Note the apparent flat and curved regions in these shapes.}
	\label{fig:nick1}
\end{figure}

As in the second row of Figure \ref{fig:nick1} the ``shape'' of $A_n$ remains consistent for large values of $n$.  This suggests the existence of a \emph{limiting} shape.  Note also several features of these shapes.  In each case there are clear ``flat'' regions.  For large $p$ these flat regions are restricted to directions in a neighbourhood of the positive orthant.  For small $p$ there is an additional flat region in directions in a neighbourhood of $(-1,-1)$.

The main results of this paper address both the law of large numbers for the chemical distance (Theorem \ref{thm:1} below) and features of the corresponding ``shape'' (Theorem \ref{thm:features} below).  We will utilise several facts about site percolation and oriented site percolation on $\Z^d$.  Let $p_*(d)$ denote the critical parameter for site percolation on $\Z^d$, and $p_{\dagger}(d)$ denote the critical parameter for \emph{oriented} site percolation on $\Z^d$.

Our first main result is the following theorem.
\begin{THM}
	\label{thm:1}
	Fix $d\ge 2$.  For $p$ such that $1-p>\min\{1-p_c(2),p_*(d)\}$ there exists a function $\zeta_p:\Z^d\to \R_+$ such that for each $v\in \Z^d$
	\[\lim_{n \to \infty}n^{-1}T_{nv}(p) =\zeta_p(v), \quad \text{almost surely and in $L^1$}.\]
\end{THM}
The claim with $v=o$ is trivial with $\zeta_p(o)=0$.  For other $v$ this result follows from an application of the subadditive ergodic theorem.  The non-obvious condition to check is that the expected passage times are finite.  For $1-p>1-p_c(2)$ we verify this via a novel (to the best of our knowledge) argument based on duality with an oriented percolation model in 2 dimensions.  For $1-p>p_*(d)$ we (roughly speaking) verify this condition by showing that to reach one vertex from another one can mostly follow a path in a supercritical percolation cluster of full sites.

For $u\in \Q^d$, let $m_u=\inf\{m \in \N:mu\in \Z^d\}$, (so  $m_u=1$ if $u\in \Z^d$).   We extend the domain of $\zeta_p$ to $\Q^d$ by defining 
	\begin{equation}
\zeta_p(u)=\zeta_p(m_u u)/m_u.
\label{zetadef}
	\end{equation}
	  It is then immediate from Theorem \ref{thm:1} that for all $u\in \Q^d$,
\begin{equation*}
\lim_{n \to \infty}(nm_u)^{-1}T_{nm_uu}(p)=\zeta_p(u), \quad \text{ almost surely and in $L^1$}.
\end{equation*} 

For $y\in \R^d$, let $[y]\in \Z^d$ denote the unique lattice point such that $y\in [y]+[0,1)^d$.  For $v\in \Z^d \setminus \{o\}$ let  $D_n(v;p)=\sup\{k\in [0,\infty): T_{[kv]}(p)\le n\}$.  Then by examining the correspondence between events of the form $\{n^{-1}D_n(v)<c\}$ and  $\{T_{[ncv]}>n\}$ one obtains the following.
\begin{COR}
\label{cor:D}
For $1-p>\min\{1-p_c(2),p_*(d)\}$, and $v\in \Z^d \setminus \{o\}$, 
\[\lim_{n \to \infty}\frac{D_n(v;p)}{n}= \zeta_p(v)^{-1}, \text{ a.s.~and in $L^1$}.\] 
\end{COR}
For fixed dimension $d\ge 2$ and any orthogonal set $O\subset \mc{E}:=\{\pm e_i:i \in [d]\}$ with $|O|=d$ (i.e.~the elements of $O$ form a basis for $\Z^d$), we can consider \emph{oriented site percolation on $\Z^d$} with orientation given by $O$.  In  other words there are (directed) connections from occupied sites $x$ to neighbouring occupied sites of  the form $x+e$ for each $e\in O$.  These oriented percolation models are identical in law, except for the change of orientation.  It is known (see e.g.~\cite{Dur84,DG,IS}) for supercritical oriented site percolation (i.e.~with parameter $q>p_*(d))$ with orientation $O$, that conditional on the percolation cluster of the origin being infinite, the cluster viewed from far away is a deterministic cone with axis $\sum_{e \in O}e$.  We call this cone  $O_q$.  As $q\uparrow 1$ this cone increases to the whole orthant $\{\sum_{e\in O}\lambda_ee:\lambda_e\ge 0,  \forall e\in O\}$.  Let
\[\mc{O}=\{O\subset \mc{E}:O \text{ is a basis for }\Z^d\}.\]

Let $\DD=\Q^d \setminus \{o\}$, and let 
\begin{align}\sg=\DD\cap \Bigg\{\sum_{i=1}^d (\alpha_i-\beta_i)e_i
: \quad & a\in [p,1], \quad 0\le \alpha_i,\beta_i \;\forall i\in [d], \text{ and }\nn\\
&\sum_{i=1}^d \alpha_i=a=1-\sum_{i=1}^d\beta_i, \text{ and }\nn\\
& \alpha_i\beta_i=0 \text{ for each  }i \in [d]\Bigg\}.
\label{constraint3}
\end{align}

Then $\sg$ is a subset of the boundary of the $\ell_1$ ball  and $\sg$ contains the intersection of this boundary  and $\DD_+=\{x\in \DD:x^{[i]}\ge 0 \text{ for every }i\}$  for every $p$ (take $a=1$).  
See Figure \ref{fig:sg} for an example in 3 dimensions.

Let $\mc{S}_p=\{u \in \DD:\zeta_p(u)=\|u\|_1\}$.    
Our second main result reveals  features of the shape including specifying other directions in $\mc{S}_p$ and its complement.

\begin{THM}
	\label{thm:features}
Fix $d\ge 2$, and $1-p>\min\{1-p_c(2),p_*(d)\}$.  Then the function $\zeta_p:\Q^d\to \R$ in Theorem \ref{thm:1} and \eqref{zetadef} satisfies the following:
\begin{itemize}
\item[(a)]
\begin{itemize}
			\item[(i)] for $q\in \Q_+$, and $u\in \Q^d$, $\zeta_p(qu)=q\zeta_p(u)$,
	\item[(ii)] $\zeta_p$ is subadditive, i.e.~$\zeta_p(u+v)\le \zeta_p(u)+\zeta_p(v)$, 
\item[(iii)] 		$\zeta_p$ is uniformly continuous. 
\end{itemize}
\item[(b)] For each fixed $v\in \Z^d$, $\zeta_p(v)$ is non-decreasing in $p$.
\item[(c)] For  $u\in \DD$  we have $u \in \mc{S}_p$ in the following cases:
\begin{enumerate}
	\item[(i)] $u\in \sg$,
\item[(ii)] $1-p>p_{\dagger}(d)$ and $u\in \bigcup\limits_{O\in \mc{O}}O_{1-p}$.
\end{enumerate}
\item[(d)] For  $u\in \DD$ with $\|u\|_1=1$  we have $u \notin \mc{S}_p$ (i.e. $\zeta_p(u)>1$) in the following cases:
\begin{enumerate}
	\item[(i)] for any $i=1,\dots,d$, $\|u+e_i\|_1<\vep$ for $\vep$ sufficiently small depending on $d,p$,
	\item[(ii)] $d=2$ and  $u=(-(1-s),s)$ for $s\in [0,p)\cap \Q$. 
\end{enumerate}
\end{itemize}
\end{THM}
To the best of our knowledge, the novel parts of this theorem (or more precisely, its proof) are (c)(i) and (d)(ii).   We expect that the other parts are standard, or at least well understood (see e.g.~\cite{ADH}).   However we have not seen a result such as (c)(ii) for dimensions  $d>2$.  To prove (c)(ii) we will use a supercritical oriented percolation fact, Lemma \ref{lem:OPhits} below, which is well known in two dimensions, but we have not found a statement of this result elsewhere when $d>2$.
 
\begin{REM}
	Part (a)(iii) of the Theorem allows one to extend the domain of $\zeta_p$ to $\R^d$ by taking limits through rationals.  It then follows from part (a) that the resulting function $\zeta_p$ is convex.
	\end{REM}

\begin{figure}
	\begin{center}
\begin{tikzpicture}[thick,scale=5]
\coordinate (A1) at (0,0);
\coordinate (A2) at (0.6,0.2);
\coordinate (A3) at (1,0);
\coordinate (A4) at (0.4,-0.2);
\coordinate (B1) at (0.5,0.5);
\coordinate (B2) at (0.5,-0.5);

\begin{scope}[thick,dashed,opacity=0.6]
\draw (A1) -- (A2) -- (A3);
\draw (B1) -- (A2) -- (B2);
\end{scope}
\draw[fill=cof,opacity=0.6] (A1) -- (A4) -- (B1);
\draw[fill=pur,opacity=0.6] (A1) -- (A4) -- (B2);
\draw[fill=greeo,opacity=0.6] (A3) -- (A4) -- (B1);
\draw[fill=greet,opacity=0.6] (A3) -- (A4) -- (B2);
\draw (B1) -- (A1) -- (B2) -- (A3) --cycle;
\end{tikzpicture}
\hspace{1cm}
\begin{tikzpicture}[thick,scale=5]
\coordinate (A1) at (0,0);
\coordinate (A2) at (0.6,0.2);
\coordinate (A3) at (1,0);
\coordinate (A4) at (0.4,-0.2);
\coordinate (B1) at (0.5,0.5);
\coordinate (B2) at (0.5,-0.5);

\coordinate (B1A3) at (0.55,0.45);
\coordinate (A4A3) at (0.46,-0.18);
\coordinate (A4B2) at (0.41,-0.23);
\coordinate (A1A2) at (0.06,0.02);
\coordinate (B1A2) at (0.51,0.47);
\coordinate (A1B2) at (0.05,-0.05);

\draw (B1A3)--(A4A3)--(A4B2)--(A1B2);

\draw[dashed,thick,opacity=0.6] (A1)--(A1A2);
\draw[dashed,thick,opacity=0.6]  (A1) -- (A2) -- (A3);
\draw[dashed,thick,opacity=0.6]  (B1) -- (A2) -- (B2);
\begin{scope}[thick,dotted,opacity=0.6]
\draw (A4A3) -- (A1A2) -- (B1A2)--(B1A3);
\end{scope}
\draw[fill=cof,opacity=0.6] (A1) -- (A4) -- (B1)--cycle;
\draw[opacity=0.6] (A1) -- (A4) -- (B1)--cycle;
\draw[fill=orange,opacity=0.6] (A1) -- (B1) -- (B1A2)--(A1A2)--cycle;
\draw[fill=gray,opacity=0.6] (A1) --(A1A2)--(A1B2)--cycle;
\draw[fill=pur,opacity=0.6] (A1) -- (A4) -- (A4B2)--(A1B2)--cycle;
\draw[fill=black,opacity=0.6]  (B1) -- (B1A3)--(B1A2)--cycle;
\draw[fill=greeo,opacity=0.6]  (A4) -- (B1)--(B1A3)--(A4A3)--cycle;
\draw[fill=greet,opacity=0.6] (A4B2) -- (A4) -- (A4A3)--cycle;
\end{tikzpicture}
\end{center}
\caption{An illustration of the region $\sg$ in 3 dimensions.  The $\ell_1$ ball in 3 dimensions appears on the left, with an example of $\sg$ depicted on the right. }
\label{fig:sg}
\end{figure}
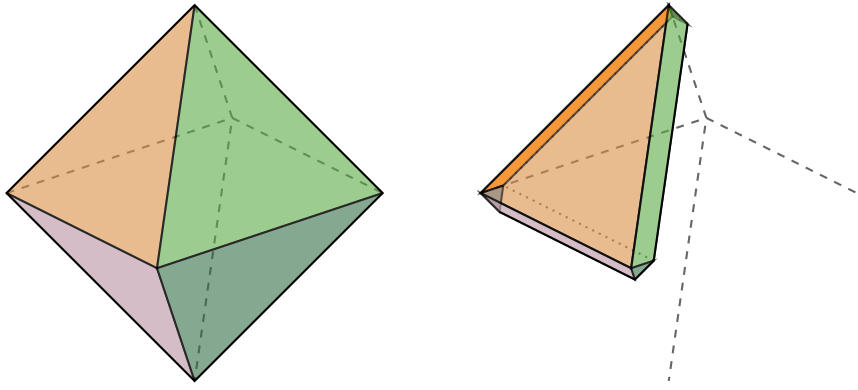

\begin{REM}[Backward connections]
\emph{One can also ask about backward clusters and connections to the origin.  
Let $\mc{B}_o$ denote the set of vertices from which the origin can be reached.  This is called the \emph{backward cluster} of the origin.  This set contains the negative orthant.  Let $\hat{T}_v=T_{v,o}$.  The subadditive ergodic theorem also applies to $\hat{T}_v$, and since $\hat{T}_v$ has the same distribution as $T_{-v}$ one can immediately infer a limit theorem (a.s.~ and in $L^1$) for $\hat{T}_v$ with ``shape'' given by $-\zeta_p(v)$.  }
\end{REM}

The remainder of this paper is organised as follows.  In Section \ref{sec:open} we present several open problems.  In Section \ref{sec:subadd} we prove Theorem \ref{thm:1}, by invoking the subadditive ergodic theorem.  The novelty therein is the proof of finiteness of the expectations in Sections \ref{sec:expected1} and \ref{sec:expected2}, while the consequences in Section \ref{sec:consequences} are standard.  The proof of Theorem \ref{thm:features} appears in Section \ref{sec:shape_props}.

\section{Open problems}
\label{sec:open}
Several open problems for the half-orthant model and other similar models are discussed e.g.~in \cite{DREphase,Shape}, and we do not repeat those here.  Instead we present some open problems regarding chemical distance  and related properties.   We start with the half-orthant model and then discuss other degenerate random environments.

\subsection{Half-orthant model}
\label{sec:open_half_orthant}
For dimensions $d>2$, Theorem \ref{thm:1} is not believed to be sharp.
\begin{OP}
Prove that the conclusion of \eqref{thm:1} holds when $p<p_c(d)$ (or $p<p'_c(d)$).
\end{OP}

Theorem \ref{thm:features} asserts directions $u$ in which vertices of the form $nm_uu$ can be reached from the origin in time of order $n\|m_u u\|_1$ (part (c) of the theorem), as well as directions in which this does not hold (part (d) of the theorem).  The two sets of directions do not exhaust all directions in $\DD$.  This observation motivates the following question.
\begin{OP}
	Does the subset of $\mc{S}_p$ presented in \ref{thm:features}(c) together with the positive orthant in fact contain all directions in $\mc{S}_p$? 
	\end{OP}

For $d\ge 2$ and $p>p_c(d)$, not all vertices are reachable from the origin.  Indeed in this setting, for ``most''  directions only finitely many points in that direction are reachable from the origin, so there is no interesting shape theorem in these directions.  On the other hand, as long as $p$ is strictly smaller than 1 there exists a neighbourhood of the positive orthant, such that for directions $u$ in this neighbourhood, $\limsup_{n \to \infty}(m_un)^{-1}T_{nm_uu}<\infty$.  For  directions in $\mc{S}_p^{\textrm{good}}$ one can show that $\limsup_{n \to \infty}(m_un)^{-1}T_{nm_uu}=1$.  However, for all $u\in \DD\setminus\DD_+$ it is the case that with positive probability $nm_u u$ is not reachable from the origin.  This means that the subadditive ergodic theorem (as stated herein) is not applicable.  Nevertheless it should be true that a shape theorem holds in directions for which (a.s.) $T_{nm_uu}$ is finite for all $n$ sufficiently large. 
\begin{OP}
	\label{op:half_orthant1}
Prove that $n^{-1}T_{nv}(p)\to \zeta_p(v)>0$ a.s.~for the half-orthant model in a non-trivial subset of directions when $p>p_c(2)$. 
\end{OP}
\begin{figure}
	\begin{center}
		\includegraphics[trim={1.5cm 1.5cm 1.5cm 1.5cm},clip,width=0.49\textwidth]{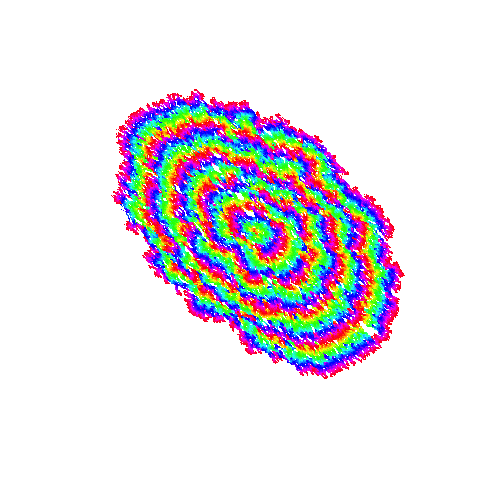}
		\includegraphics[trim={2cm 2cm 0 0},clip,width=0.49\textwidth]{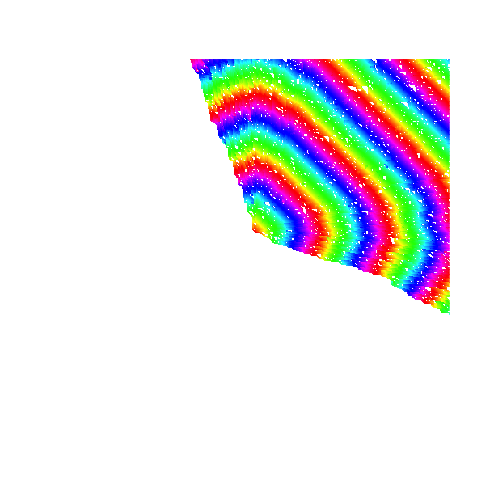}
	\end{center}
	\vspace{-1cm}
	\caption{Simulation of the set of points reachable in $n$ steps for the orthant model with $p=1/2$ (left) and $p=3/4$ (right).  White patches surrounded by colour are not reachable from the origin.}
	\label{fig:orthant}
\end{figure}
\subsection{Orthant model}
\label{sec:open_orthant}

For the orthant model (where full sites are replaced with half-sites of opposite orientation), in any direction $u\in \mathbb{Q}^2\setminus \{o\}$ there are a.s.~infinitely many points that are not reachable from the origin.  Thus the limit as presented in Theorem \ref{thm:1} fails in this model in every direction.  Nevertheless it is clear from e.g.~Figure \ref{fig:orthant} that there is still a shape to the set $A_n$ when $n$ is large.  This setting is similar to the setting of  first-passage percolation when passage times across edges can be infinite with positive probability (see e.g.~\cite{CerfTheret}).
\begin{OP}
	\label{op:orthant1}
Prove a shape theorem (e.g.~of the form in Corollary~\ref{cor:D}) for the orthant model in dimensions $d\ge 2$ in all directions when $p\in (1-p_c(2),p_c(2))$.  
\end{OP}
The next open problem is a version of Open Problem \ref{op:half_orthant1} suitable for the orthant model.  
\begin{OP}
	\label{op:orthant2}
	Prove a shape theorem (e.g.~of the form in Corollary \ref{cor:D}) for a non-trivial subset of directions when $p\notin  (1-p_c(2),p_c(2))$.  
\end{OP}

\subsection{Other models}
\label{sec:open_other}

Finally, there are many more examples of degenerate random environments where shape theorems are relevant and interesting.  An example that has a similar flavour to what has been presented in this paper is the case where we have just the arrow $e_1$ with probability $p$ and full sites otherwise.  
\begin{OP}
	\label{op:ogeneral}
	Investigate shape theorems in more general degenerate random environments. 
\end{OP}

All of the models mentioned above are partially oriented models of random media.  In the unoriented setting (the standard model being first passage percolation) one is also interested in features of shortest paths (or \emph{geodesics}) other than just their length.  One can ask similar questions in the setting of degenerate random environments.

\begin{OP}[Geodesics]
	\label{op:geodesics}
Investigate properties of the geodesics in some of these models.
\end{OP}
In this paper all edge weights (or transit times) have been set to 1.  One can generalise further to settings where the directed edges can also have random transit times associated to them.  This arises naturally from the study of random walks in (non-elliptic) random environments.  For example, one could take the transit time of a directed edge $e$ with weight $w(e)\in (0,1]$ as $1/w(e)$, with infinite transit time when $w(e)=0$.
\begin{OP}[Edge weights]
	\label{op:weights}
Generalise to situations where general weights are allowed on the directed edges.
\end{OP}

\section{Proof of Theorem \ref{thm:1}}
\label{sec:subadd}
Fix $v\in \Z^d \setminus \{o\}$, and let $X_{m,n}=T_{mv,nv}$ denote the minimum number of steps to reach $nv$, starting from $mv$.  We show that $X_{m,n}$ satisfies the assumptions of the subadditive ergodic theorem.
\begin{THM}[Kingman, Liggett]
	Let $(X_{m,n})_{0\le m<n}$ be a family of random variables satisfying the following:
	\begin{itemize}
		\item[(i)] $X_{0,n}\le X_{0,k}+X_{k,n}$ for all $0< k<n$;
		\item[(ii)]  the distribution of the sequence $(X_{m,m+k})_{k\ge 1}$ does not depend on $m$;
		\item[(iii)] for each $n$, $\E[|X_{0,n}|]<\infty$ and $\E[X_{0,n}]\ge cn$ for some constant $c>-\infty$;
		\item[(iv)] for each $k$, $(X_{nk,(n+1)k})_{n\ge 1}$ is a stationary and ergodic sequence.
	\end{itemize}
Then $X=\lim_{n \to \infty} X_{0,n}/n$ exists a.s. and in $L^1$, and $X$ is deterministic.
	\end{THM}
We apply this theorem to the random variables $X_{m,n}=T_{mx,nx}$ for arbitrary but  fixed $x\in \Z^d\setminus \{o\}$.  Conditions (ii) and (iv) in the theorem follow from the fact that the environment is i.i.d., while condition (i) is also straightforward since if $\gamma'$ is a path from $o$ to $kx$ that is consistent with the environment and $\gamma''$ is a path from $kx$ to $nx$ that is consistent with the environment then the concatenation of the two paths yields a path $\gamma$ from $o$ to $nx$ that is consistent with the environment.  Trivially $\E[T_{a,b}]\ge 0$ since $T_{a,b}\ge 0$.  Therefore, in order to apply the theorem, it remains to prove that $\E[|T_{o,nx}|]<\infty$ for every $n$.  Since we also want this for every $x$, our goal is to show that $\E[T_{o,x}]<\infty$ for every $x$.  Since every $x$ is a sum of $\|x\|_1$-many vectors $e\in \mc{E}$, it is in fact sufficient to show that $\E[T_{o,e}]<\infty$ for every $e\in \mc{E}$.  

For the half-orthant model, $T_{o,e}=1$ if $e\in \mc{E}_+=\{e_i:i \in [d]\}$ and by symmetry $\E[T_{o,-e_j}]=\E[T_{o,-e_1}]$ so we need only show that $\E[T_{o,-e_1}]<\infty$.  We will verify this in the following two sections (for fixed $d\ge 2$), in the cases $p<p_c(2)$ and $1-p>p_*(d)$ respectively.

\subsection{Finite expectation for $p<p_c(2)$}
\label{sec:expected1}

It is sufficient to prove the result in the case $d=2$, since in higher dimensions the time $T_{o,-e_1}$ is less than or equal to the time it takes to reach $-e_1$ from $o$ using only moves that stay in $\Z^2 \times \{0\}^{d-2}$.  So for the remainder of this section we restrict our attention to $d=2$.

Define a random set $\mathfrak{C}\subset \DD_{-,+}:=\{x\in \Z^2:x^{\sss[1]}\le 0, x^{\sss[2]}\ge 0\}$ as follows.  Let $\mathfrak{C}_0=\{o\}=\{(0,0)\}$.  

If the origin is a full site then let $\mathfrak{C}=\mathfrak{C}_0=\{o\}$.  

Otherwise the origin is a half site.  In this case let $\mathfrak{C}_1$ denote the set of $x\in \DD_{-,+}$ with $\|x\|_1=1$ such that $x$ is a half site and $\{x+e_1,x-e_2\}\cap  \mathfrak{C}_0\ne\varnothing$.  For $n\ge 2$, we define $\mathfrak{C}_n$ to be the set of $x\in \DD_{-,+}$ with $\|x\|_1=n$ such that $x$ is a half site and either $\{x+e_1,x-e_2\}\cap  \mathfrak{C}_{n-1}\ne\varnothing$ or $x+e_1-e_2\in  \mathfrak{C}_{n-2}$.  Let $\mathfrak{C}=\cup_{n\ge 0}\mathfrak{C}_n$.   

It follows that $\mathfrak{C}$ is the cluster of the origin for an oriented site percolation model on the triangular lattice with parameter $p$; sites are occupied if and only if they are half-sites, and the connections are in the three directions $-e_1,+e_2,-e_1+e_2$ of a triangular lattice.  If $p<p_c(2)\approx 0.59$ then since $p_c(2)$ is also the critical value for oriented site percolation on the triangular lattice \cite{DRE}, the cluster $\mathfrak{C}$ is finite. Thus $K:=\sup\{k:\mathfrak{C}_k\ne \varnothing\}<\infty$ almost surely.  Moreover (see e.g.~\cite[Section 7 (6)]{Dur84}) there exists $c(p)>0$ such that for all $n$, 
\begin{equation}
\label{eqn:size_of_tri_cluster}
	\P(|\mathfrak{C}(p)|>n)\le e^{-c(p)n}.
\end{equation}  

Define the set of \emph{awesome} points $\mathbb{A}$ by 
\begin{multline*}
	\mathbb{A}:=\{x\in \DD_{-,+}\cap \mathfrak{C}: x\text{ is a half site} \\ \text{but none of $x-e_1,x+e_2,x-e_1+e_2$ is a half site}\}.
\end{multline*}
Note that if $K\ge 1$ then $\mathfrak{C}_K\subset \mathbb{A}$  since every point in $\mathfrak{C}_K$ is a half site and $\mathfrak{C}_{K+1}$ and $\mathfrak{C}_{K+2}$ are empty.
\begin{LEM}
There is a path from $o$ to $-e_1$, consistent with the environment, of length at most $4|\mathfrak{C}|$.
\end{LEM}
\begin{proof}
We prove by induction on $n=|\mathfrak{C}|$ that there is a path of length at most $4n$ from $o$ to $-e_1$. 

If $|\mathfrak{C}|=1$ then either $o$ is a full site or $o$ is awesome.  In the former case there is a path of length 1 consistent with the environment $\bs{\omega}$, from $o$ to $-e_1$, consisting of the step $-e_1$.  In the latter case both $e_2$ and $e_2-e_1$ are full sites, so there is a path of length 3 consistent with the environment $\bs{\omega}$ from $o$ to $-e_1$ consisting of the steps $e_2,-e_1,-e_2$ in that order.  

Let $n\ge 1$ and assume the result is true for every possible cluster of size $n$.  Now suppose that we are given a cluster $\mathfrak{C}$ of size $|\mathfrak{C}|=n+1$.  Since this cluster is finite but non-empty, it necessarily contains an awesome point $x$ (in particular, since $n+1>1$ we have that $K\ge 1$ and  every point in $\mathfrak{C}_K$ is awesome).

Since $x$ is awesome, by changing the status of $x$ from half-site to full-site, we obtain a new environment (call it $\bs{\omega}^x$), and a new cluster $\mathfrak{C}^x$ of size $|\mathfrak{C}^x|=n$.  It follows from the induction hypothesis that there exists a path $\gamma$ consistent with the environment $\bs{\omega}^x$ from $o$ to $-e_1$ of length at most $4n$.   By deleting loops if necessary we may assume that $\gamma$ is simple, so it passes through each point at most once.  We will find a path $\gamma'$ in the original environment  $\bs{\omega}$ from $o$ to $-e_1$, by modifying $\gamma$.

Suppose that $\gamma$ does not pass through (i.e.~enter and exit) $x$.  Then the path $\gamma$ is also consistent with $\bs{\omega}$ so we take $\gamma'=\gamma$.  Otherwise $\gamma$ enters $x$ from some direction and exits from some other direction.  If $\gamma$ exits in the direction $e_1$ or $e_2$ then (since these arrows are still available from $x$ in the environment $\bs{\omega}$) $\gamma$ is still consistent with $\bs{\omega}$, so we take $\gamma'=\gamma$ again.  

If $\gamma$ exits $x$ in direction $-e_1$, then $\gamma_k=x$, and $\gamma_{k+1}=x-e_1$ for some $k< 4n$.  In this case we define $\gamma'$ by $\gamma'_j=\gamma_j$ for $j\le k$, $\gamma'_{k+1}=x+e_2$, $\gamma'_{k+2}=x+e_2-e_1$, $\gamma'_{k+3}=x+e_2-e_1-e_2$, and then $\gamma'$ follows the remainder of the path $\gamma$ from the point $x-e_1=x+e_2-e_1-e_2$.  All of the added steps are consistent with $\bs{\omega}$ since $x$ was awesome in this environment.  The resulting path $\gamma'$ (which may or may not be simple) has two extra steps than $\gamma$.

Similarly, if $\gamma$ exits $x$ in direction $-e_2$, then we set $\gamma'=\gamma$ until the hitting time of $x$, then we make $\gamma'$ take the steps $e_2,-e_1,-e_2,-e_2,e_1$ (all these steps are consistent with $\bs{\omega}$ since $x$ was awesome in $\bs{\omega}$) and then $\gamma'$ follows $\gamma$ from the point $x+e_2-e_1-e_2-e_2+e_1=x-e_2$.  In this case $\gamma'$ (which may or may not be simple) contains 4 more steps than $\gamma$.  

It follows that we have a path $\gamma'$ consistent with the environment $\bs{\omega}$, from $o$ to $-e_1$, whose length satisfies $|\gamma'|\le 4+|\gamma|\le 4+4n=4(n+1)$.  This completes the proof.
\end{proof}

We have verified that for every $x\in \Z^d$, $\E[T_x]<\infty$.   Applying the subadditive ergodic theorem then yields the following.
\begin{LEM}
	\label{lem:Zd}
The claim of Theorem \ref{thm:1} holds for $p<p_c(2)$.
\end{LEM}

\subsection{Finite expectation for $1-p>p_*(d)$}
\label{sec:expected2}
Lemma \ref{lem:Zd} and the follow lemma together imply that Theorem \ref{thm:1} holds.
\begin{LEM}
	\label{lem:Zd2}
	The claim of Theorem \ref{thm:1} holds for $1-p>p_*(d)$.
\end{LEM}

 Let $\P^*_q$ denote the law of site percolation on $\Z^d$, with parameter $q$.  
For supercritical site percolation (i.e.~with $q>p_*(d)$) let $\mc{P}$ denote the (unique) infinite cluster, and let 
$N^+=\inf\{n\ge 0:ne_1\in \mc{P}\}$.
We will use the following result, which is well-understood (but we have not seen an explicit statement in the literature).   

\begin{LEM}
\label{lem:siteperc}
Fix $d\ge 2$.  For $q>p_*(d)$ there exist $c,C>0$ such that for all $n>0$,
\[\P^*_q(N^+>n)\le Ce^{-cn}.\]
\end{LEM}
\begin{proof}(Supplied by G.~Grimmett)
First assume $d \ge 3$, and let $q>p_*(d)$.  Let $S_k = [1,k]\times \Z^{d-1}$ be the slab of thickness $k$.  By \cite{GM} there exists $k_0$ such that $q>p_*(S_{k_0})$ (the critical probability for percolation in the slab).   Let $\P^0_q$ denote the law of site percolation in this slab.  Then there exists $a>0$ 
such that $ \P^0_q(re_1 \leftrightarrow \infty) >a$ (for every $r \in [1,k_0]$).  The interval $J_n=[1,nk_0e_1]$ traverses $n$ disjoint copies of $S_{k_0}$.  The contents of these slabs are independent (since they are vertex disjoint), so that $\P^*_q(N^+>mk_0)\le \P^*_q(J_m \nleftrightarrow \infty)\le (1-a)^m$.  For $n \in ((m-1)k_0,mk_0]$ we then have
\[\P^*_q(N^+>n)\le ((1-a)^{1/k_0})^n (1-a)^{-1},\]
which proves the result.

Now take $d = 2$ and $q>p_*(2)$. 
If $I_n=[0,ne_1]$ is not joined to $\infty$, then all
points in $I_n$ are in finite open clusters. Let $\mc{J}$ be the union of these clusters, and let $\mc{K}$ be the exterior site boundary of $\mc{J}$ in $\Z^2$.  Let $L$ be the matching lattice of $\Z^2$ (i.e.~$\Z^2 $ with all diagonals added).   
It is 
``standard'' that $\mc{K}$ is connected in $L$.   Since $q>p_*(2)$, and $p_*(2)+ p_c(L) = 1$ (see e.g.~\cite{GL22}), the closed sites are subcritical for site percolation on $L$, and therefore the diameter of the closed cluster containing $(-m,0)$ has exponential tail.  Thus there exist $C,c>0$ such that for all $n$, the probability that the diameter of the (closed) cluster of $(-m,0)$ in $L$ exceeds $n$ is at most $Ce^{-cn}$.   Now note that $\mc{K}$ is (part of) a closed cluster in this subcritical site model.  Let $(-b,0)$ be the leftmost vertex of 
$\mc{K}$ on the axis. Then the closed cluster of $(-b,0)$ in $L$ has diameter
at least $b+n$.  It follows that 
$\P^*(I_n \nleftrightarrow \infty) \le \sum_{b\ge 0}  C e^{-c(n+b)}$ as required.
\end{proof}

\begin{proof}[Proof of Lemma \ref{lem:Zd2}]
Fix $d\ge 2$, and $1-p>p_*(d)$.  Then the full sites percolate in the sense of site percolation on $\Z^d$.  Let $\mc{P}$ denote the infinite cluster (in the sense of site percolation) of the set of full sites.  Let $N^+=\inf\{n\ge 0:ne_1\in \mc{P}\}$, and let $N^-=\inf\{n>0:-ne_1\in \mc{P}\}$.   Let $\Gamma_N$ denote any shortest path from $N^+e_1$ to $-N^-e_1$, and let $M=|\Gamma_N|$.  Such a path exists because there is a path from one to the other that consists only of full sites.

Now consider the path $\Gamma$ from the origin to $-e_1$ as follows.  Starting from $o$, the path takes $N^+$ steps in the direction $e_1$ to the point $N^+e_1$.  It then follows the path $\Gamma_N$ from $N^+e_1$ to $-N^{-}e_1$.  It then takes $N_--1$ steps in direction $e_1$ to reach the point $-e_1$.  This entire path is consistent with the environment and the total length is at most $N^+ + M+N^-$.  Thus, $T_{o,-e_1}\le N^+ +N^- + M$ and it suffices to show that the three random variables on the right all have finite expectation.  Finiteness of the first two expectations follows from Lemma \ref{lem:siteperc} (or indeed from 
the ergodic theorem as $\E[N^+]$ and $\E[N^-]$ are both at most $1/\P(o\in \mc{P})$).  It remains to show that $\E[M]<\infty$. 

For $n\in \N$,
\begin{align}
\P(M>n)&\le \P\Big(M>n,N^++N^-\le \sqrt{n}\Big)\label{goat-1}\\
&\quad +\P\Big(N^+>\frac{\sqrt n}2\Big)+\P\Big(N^->\frac{\sqrt n}2\Big).\label{goat0}
\end{align}
The last two terms are summable in $n$ by Lemma \ref{lem:siteperc}.   It remains to show that the first term on the right of \eqref{goat-1} is summable in $n$.
This term is at most 
\begin{align}
&\sum_{n_+=0}^{\sqrt{n}}\sum_{n_-=1}^{\sqrt{n}}
\P^*_{1-p}\Big(n_+e_1\leftrightarrow -n_-e_1,T^*_{n_+e_1,-n_-e_1}>n\Big)\nn\\
&=\sum_{n_+=0}^{\sqrt{n}}\sum_{n_-=1}^{\sqrt{n}}
\P^*_{1-p}\Big(o\leftrightarrow (n_++n_-)e_1,T^*_{o,(n_++n_-)e_1}>n\Big)\label{goat}
\end{align}
where now we are only considering connections using full sites.  We now use the argument of \cite[Theorems 1.1 and 1.2]{AP96}.\footnote{This paper proves the result for bond percolation, but also points out that the same arguments also prove the result for site percolation.}  Specifically we use the inequality \cite[(4.49)]{AP96} with the box radius $N$ therein chosen so that \cite[(4.47)]{AP96} holds, and setting $y$ in \cite[(4.49)]{AP96} equal to $me_1$.  Note that in \cite[(4.49)]{AP96} their $l$ is our $n$, while their $n$ is equal to $|\boldsymbol{a}(y)|\le |y|=m$.  So  \cite[(4.49)]{AP96} together with Markov's inequality applied to $e^{h(|\boldsymbol{C_0}^*|+1)}$ therein gives the existence of constants $C^*,c^*>0$ such that for all $m,n\in \N$,
\begin{align*}
\P^*_{1-p}(o \leftrightarrow me_1,T^*_{o,me_1}>n)\le C^*m e^{-c^* n/m}.
\end{align*}
Setting $m=n_++n_-\le 2\sqrt{n}$ it follows that there exist $a,b>0$ such that for all $n\in \N$, \eqref{goat} is at most 
\begin{align*}
an\sqrt{n} e^{-b \sqrt{n}}.
\end{align*}

This proves that $\P(M>n)$ is summable in $n$, so $\E[M]<\infty$, which completes the proof.
\end{proof}

\subsection{Consequences}
\label{sec:consequences}
Before turning to the proof of Corollary \ref{cor:D}, 
we first establish the following extension of Theorem \ref{thm:1}.  Recall that for $y\in \R^d$, $[y]\in \Z^d$ denotes the unique lattice point such that $y\in [y]+[0,1)^d$.
\begin{LEM}\label{lem:limit_through_reals}
	Fix $d\geq2$ and $1-p > \min\{1-p_c(2),p_*(d)\}$. Then for every $u\in\Q^d$,
	\begin{equation}
	\lim_{x\to\infty} \frac{T_{[xu]}}{x} = \zeta_p(u), \quad \text{almost surely, and in $L^1$,}
	\end{equation}
	where we are taking $x$ through $\mathbb{R}$ rather than just $\mathbb{Z}$.	
\end{LEM}

\begin{proof}
The claim is trivial at $u=o$.   Otherwise we first we first restrict to $\mathbb{Z}^d$ and then at the end generalise to $\mathbb{Q}^d$.	Suppose that $v\in \Z^d\setminus \{o\}$.  
	By subadditivity, 
	\begin{align}
	T_{[xv]} &\leq T_{\lfloor x \rfloor v} + T_{\lfloor x \rfloor v, [xv]} \\
	T_{\lfloor x \rfloor v} &\leq T_{[xv]} + T_{[xv], \lfloor x \rfloor v}.
	\end{align}
	Thus
	\begin{equation}
	T_{\lfloor x \rfloor v} - T_{[xv],\lfloor x \rfloor v} \leq T_{[xv]} \leq T_{\lfloor x \rfloor v} + T_{\lfloor x \rfloor v,[xv]}.
	\end{equation}
	Now note that by Theorem \ref{thm:1}
	\begin{equation}
	\frac{T_{\lfloor x \rfloor v}}{x} = \frac{\lfloor x \rfloor}{x} \cdot \frac{T_{\lfloor x \rfloor v}}{\lfloor x \rfloor} \to \zeta_p(v) \quad \text{as $x\to\infty$, a.s.~and in $L^1$.}
	\end{equation}
	Thus, to prove the claim on $\Z^d$ it remains to show that almost surely
	\begin{equation}
	\lim_{x\to\infty} \frac{T_{[xv],\lfloor x \rfloor v}}{x} = \lim_{x\to\infty} \frac{T_{\lfloor x \rfloor v,[xv]}}{x} = 0,
	\end{equation}
	and that the limits of the  corresponding expected values are also 0.
	
	We will explain the argument in detail for $d=2$ and then explain how to generalise to $d>2$. 
	Let $s_i = \sgn(v^{\sss[i]})$.
	Define the set of vertices
	\begin{equation}
	B_x(v) := \big\{\lfloor x \rfloor v +is_1e_1 + js_2e_2 : i,j \in \mathbb{Z} \text{ and } 0 \leq i \leq |v^{\sss[1]}|, 0 \leq j \leq |v^{\sss[2]}|\big\}.
	\end{equation}
	That is, $B_x(v)$ is a box of width $|v^{[1]}|$ and height $|v^{[2]}|$, positioned so that $\lfloor x \rfloor v$ is its corner closest to the origin.
	
	Since
	\begin{equation}
	\lfloor x \rfloor |v^{\sss[i]}| \leq x |v^{\sss[i]}| \leq (\lfloor x \rfloor+1) |v^{\sss[i]}|
	\end{equation}
	we can conclude that $[xv] \in B_x(v)$. We thus have the bound
	\begin{align}
	\frac{T_{\lfloor x \rfloor v,[xv]}}{x} &\leq  \frac{\max\{T_{\lfloor x \rfloor v,u} : u \in B_x(v)\}}{x} \\
	&\leq \sum_{\substack{0 \leq i \leq |v^{\sss[1]}| \\ 0 \leq j \leq |v^{\sss[2]}|}} \frac{T_{\lfloor x\rfloor v, \lfloor x \rfloor v +is_1e_1 + js_2e_2}}{x}.\label{donkey1}
	\end{align}
	Similarly
	\begin{equation}\label{eqn:Txv_box_sum}
	\frac{T_{[xv],\lfloor x \rfloor v}}{x} \leq \sum_{\substack{0 \leq i \leq |v^{\sss[1]}| \\ 0 \leq j \leq |v^{\sss[2]}|}} \frac{T_{\lfloor x \rfloor v +is_1e_1 + js_2e_2,\lfloor x\rfloor v}}{x}.
	\end{equation}
	
	By translation invariance we have $T_{u,u+v} \stackrel{d}{=} T_v$ for $u,v \in \mathbb{Z}^d$. Thus for any $n \in \mathbb{Z}_+$ and $\varepsilon>0$ with fixed $0\leq i \leq |v^{[1]}|, 0 \leq j \leq |v^{[2]}|$ (excluding $i=j=0$, which is trivial)
	\begin{align}
	\mathbb{P} \Big( \frac{T_{nv,nv+is_1e_1+js_2e_2}}{n} > \varepsilon\Big) &= \mathbb{P}(T_{is_1e_1+js_2e_2} > n\varepsilon) \\
	&\leq \mathbb{P}(T_{is_1e_1} > n\varepsilon\cdot \textstyle\frac{i}{i+j}) + \mathbb{P}(T_{is_1e_1,is_1e_1+js_2e_2} > n\varepsilon\cdot \frac{j}{i+j}) \\
	&= \mathbb{P}(T_{is_1e_1} > n\varepsilon\cdot \textstyle\frac{i}{i+j}) + \mathbb{P}(T_{js_2e_2} > n\varepsilon\cdot \frac{j}{i+j}) \\
	&\textstyle\leq i \mathbb{P}(T_{s_1e_1} > \frac{n\varepsilon}{i+j}) + j\mathbb{P}(T_{s_2e_2} > \frac{n\varepsilon}{i+j}).
	\end{align}
	If $s_i = 0$ then of course $T_{s_ie_i} = 0$, while if $s_i=1$ then $T_{s_ie_i} = 1$.  Otherwise we have $s_i = -1$. Note that $\sum_{n \geq 1} \mathbb{P}(T_{-e_1} > n) \leq \mathbb{E}[T_{-e_1}] < \infty$ by Lemmas~\ref{lem:Zd} and~\ref{lem:Zd2}. Then because
$X=T_{-e_1}(i+j)/\vep\ge 0$ we have 	
		\begin{equation*}
\sum_{n\geq 1}\P\Big(T_{-e_1} > \frac{n\vep}{i+j}\Big) 
=\sum_{n\geq 1} \P(X>n)\le \E[X]= \frac{i+j}{\vep}\cdot \E[T_{-e_1}]<\infty.
		\end{equation*}

	 With $E_n$ denoting the event $\{\frac{T_{nv,nv+is_1e_1+js_2e_2}}{n} > \varepsilon\}$, it follows that
	\begin{equation}
	\sum_{n=1}^\infty \mathbb{P}(E_n) < \infty,
	\end{equation}
	and then by the Borel-Cantelli lemma $\mathbb{P}(E_n \text{ occurs infinitely often}) = 0$.  Since this is true for any $\varepsilon>0$, we must have
	\begin{equation}
	\lim_{n\to\infty} \frac{T_{nv,nv+is_1e_1+js_2e_2}}{n} = 0, \quad\text{a.s.}
	\end{equation}
	
	Switching back to real $x$, we have that almost surely
	\begin{equation}
	\lim_{x\to\infty} \frac{T_{\lfloor x \rfloor v,\lfloor x \rfloor v+is_1e_1+js_2e_2}}{x} = \lim_{x\to\infty}\frac{T_{\lfloor x \rfloor v,\lfloor x \rfloor v+is_1e_1+js_2e_2}}{\lfloor x \rfloor} \cdot \frac{\lfloor x \rfloor}{x} = 0.
	\end{equation}
	Applying this to \eqref{eqn:Txv_box_sum} shows that
	\begin{equation}
	\lim_{x\to\infty} \frac{T_{\lfloor x \rfloor v,[xv]}}{x} = 0, \quad\text{a.s.}
	\end{equation}
	Similar arguments show the corresponding result for $T_{[xv],\lfloor x \rfloor v}$.   Expectations are handled similarly.  Take the expectation of both sides of \eqref{donkey1} (resp.~\eqref{eqn:Txv_box_sum}) and note that  the expectations of the numerator terms on the right hand side of \eqref{donkey1} are equal to $\E[T_{is_1e_1+js_2e_2}]<\infty$.   Since the sums are finite and $x \to \infty$ we see that $\E[\frac{T_{\lfloor x \rfloor v,[xv]}}{x}] \to 0$ as $x\to \infty$ and similarly $\E[	\frac{T_{[xv],\lfloor x \rfloor v}}{x} ]\to 0$.
	This completes the proof for $d=2$. 
	
	The generalisation to $d>2$ is straightforward: the box $B_x(v)$ is just the higher-dimensional analogue, with the point $\lfloor x \rfloor v$ the corner closest to the origin. Then again $[xv] \in B_x(v)$, and to get from $\lfloor x \rfloor v$ to $[xv]$ we again get an upper bound by considering each dimension separately.
	
To upgrade the result to $u\in \Q^d$ simply note that $v=m_u u\in \Z^d \setminus \{o\}$, and 
\begin{align*}
\frac{T_{[xu]}}{x}=\frac{1}{m_u}\frac{T_{[x/m_u \cdot v]}}{x/m_u}.
\end{align*}
As $x\to \infty$ also $x/m_u\to \infty$, and the limit of the above is therefore $\zeta_p(v)/m_u=\zeta_p(u)$ as required.
\end{proof}

\begin{proof}[Proof of Corollary \ref{cor:D}]
	
	Fix $v\in \Z^d$, and let $\vep>0$.  We show that almost surely 
	\begin{align}
	\liminf_{n \to \infty}n^{-1} D_n(v)&\ge \dfrac{1}{(1+\vep)\zeta_p(v)}, \qquad \text{ and}\label{liminfD}\\
	\limsup_{n \to \infty}n^{-1} D_n(v)&\le \dfrac{1}{(1-\vep)\zeta_p(v)}.\label{limsupD}
	\end{align}
	This proves that $n^{-1} D_n(v)\to \zeta_p(v)^{-1}$ almost surely and the $L^q$ convergence (for any $q>0$) follows from dominated convergence since $D_n(v)/n\in [0,1/\|v\|]$ for $v\in \Z^d$.	 
	
	To prove \eqref{liminfD}, 
	let $c_1:=\dfrac{1}{(1+\vep)\zeta_p(v)}$.  Then, a.s., for all but finitely many $n$ we have (by Lemma \ref{lem:limit_through_reals})
	\[\dfrac{T_{[nc_1v]}}{nc_1}\le \frac{1}{c_1}.\]
	So, a.s., for all but finitely many $n$ we have
	$T_{[nc_1v]}\le n$ and  therefore also $D_n(v)\ge nc_1$ as required.
	
	To prove 	\eqref{limsupD} let $c_2=\dfrac{1}{(1-\vep)\zeta_p(v)}$, and suppose that 
	\[\limsup_{n \to \infty}n^{-1} D_n(v)> c_2.\]
Then there exists a random subsequence $(N_k)_{k \in \N}$ with $N_k\ge N_{k-1}+1$ such that 	
	\[\dfrac{D_{N_k}(v)}{N_k}> c_2, \quad \text{ for every }k.\]
	Setting $M_k=D_{N_k}(v)$ it follows that for every $k$, 
	\begin{align}
	\sup\{m:T_{[mv]}\le N_k\}=M_k>c_2N_k.\label{hohum}
	\end{align}
In particular, for every $k$ there exists $M'_k\in (c_2N_k,M_k]$ such 	that $ T_{[M'_k v]}\le N_k$.
	
Since by Lemma \ref{lem:limit_through_reals}, 
$x^{-1}T_{[xv]}\to \zeta_p(v)$ a.s.~as $x\to \infty$ we have that $T_{[xv]}\ge (1-\vep/2)\zeta(v)x$ for all  $x$ sufficiently large (a.s.).  In particular, for all $k$ sufficiently large we have 
\[T_{[M'_kv]}>(1-\vep/2)\zeta(v)M'_k>(1-\vep/2)\zeta(v)c_2N_k>N_k,\]
which contradicts $T_{[M'_kv]}\le N_k$ above.  Thus the event that $\limsup_{n \to \infty}n^{-1} D_n(v)> c_2$ has probability 0, which completes the proof.
	\end{proof}

\section{Features of the shape for $p<p_c(2)$}
\label{sec:shape_props}
In this section we prove Theorem \ref{thm:features}, except for (b) which is omitted as the proof is via an elementary (and standard) coupling argument.
\subsection{Proof of Theorem \ref{thm:features}(a)}
\begin{proof}[Proof of Theorem \ref{thm:features}(a)(i)]
	This follows from Lemma \ref{lem:limit_through_reals} since
\[\zeta_p(qu)=\lim_{x \to \infty}\frac{T_{[xqu]}}{x}=q \lim_{x \to \infty} \frac{T_{[xqu]}}{xq}=q\zeta_p(u).\]
\end{proof}

\begin{proof}[Proof of Theorem \ref{thm:features}(a)(ii)]
By Theorem \ref{thm:1} and subadditivity for $T$, for $x,y\in \Z^d$ we have 
	\begin{align*}
	\zeta_p(x+y) &= \lim_{n \to \infty} \frac{\mathbb{E}[T_{n(x+y)}]}{n} \\
	&\leq \lim_{n \to \infty} \frac{\mathbb{E}[T_{nx} + T_{nx,n(x+y)}]}{n}\\
	&= \lim_{n \to \infty} \frac{\mathbb{E}[T_{nx}] + \mathbb{E}[T_{ny}]}{n}\\
	&= \zeta_p(x) + \zeta_p(y).
	\end{align*}
For  $x,y\in \Q^d$  there exists an $m=m(x,y)\in \N$ such that $mx,my\in \Z^d\setminus\{o\}$ and then from (i) 
\begin{align*}
\zeta_p(x+y)=\zeta_p(m^{-1}(mx+my))&=m^{-1}\zeta_p(mx+my)\\
&\le m^{-1}(\zeta_p(mx)+\zeta_p(my))=\zeta_p(x)+\zeta_p(y).
\end{align*}
	\end{proof}
	
	\begin{proof}[Proof of Theorem \ref{thm:features}(a)(iii)]
Let $x,h \in \Q^d$. From (ii) we have
	\begin{equation}
	\zeta_p(x+h) \leq \zeta_p(x) + \zeta_p(h), \qquad \zeta_p(x) \leq \zeta_p(x+h) + \zeta_p(-h),
	\end{equation}
	from which it follows that
	\begin{equation}
	|\zeta_p(x+h)-\zeta_p(x)| \leq \max\{\zeta_p(h),\zeta_p(-h)\}.
	\end{equation}
	Now we always have $\zeta_p(e_i) = 1$ and $\zeta_p(-e_i) \geq 1$. Thus, using (i) and (ii),
	\begin{align*}
	\zeta_p(h) &= \zeta_p\bigg(\sum_{i=1}^d h^{[i]}e_i\bigg) \leq \sum_{i=1}^d \zeta_p(h^{[i]}e_i) \leq \sum_{i=1}^d |h^{[i]}| \zeta_p(-e_i)= \|h\|_1 \zeta_p(-e_1).
	\end{align*}
	The same upper bound applies to $\zeta_p(-h)$. Hence
	\begin{equation}
	|\zeta_p(x+h)-\zeta_p(x)| \leq \|h\|_1 \zeta_p(-e_1).
	\end{equation}
	This verifies the claim.
	\end{proof}

\subsection{Proof of Theorem \ref{thm:features}(c)}

 \begin{proof}[Proof of Theorem \ref{thm:features}(c)(i)]
 Fix $d\ge 2$ and $p$ as in the theorem.  
  Note first that by continuity of $\zeta$ it is sufficient to prove the result for $v$ with rational coordinates, and with $a\in (p,1)$.  
 	
Fix $v\in \sg$ (with rational coordinates).  Then we can write $v$ as   
\[v=\sum_{i=1}^d (a\alpha'_i-(1-a)\beta'_i)e_i,\]
where $\alpha'_i=\alpha_i/a$ and $\beta'_i=\beta_i/(1-a)$.
Since $v$ has rational coordinates it follows that  $m_v:=\inf\{k\in \N:kv\in \Z^d\}<\infty$, and  the points $nm_vv$ are in $\Z^d$ for $n\in \Z_+$.  It therefore suffices to show that there is a constant $C>0$  such that almost surely for each $\vep>0$, $T_{m_vnv}\le nm_v+C\vep n$ infinitely often.  We will prove this by constructing a path $\Gamma$ consistent with the environment along which every site $x$ is reachable in time $\|x\|_1$, and such that the points $m_vnv$ are reachable by following $\Gamma$ for a long time and then taking $+e_i$ steps for a relatively short time.
 
 We start by enlarging the probability space to include the random environment, as well as some i.i.d.~standard uniform random variables $(U_n)_{n\ge 0}$ that are also independent of the environment.  Proving the result in this enlarged space establishes the desired result (since the desired result is simply an a.s.~statement about the environment).

Let $\vep<a-p$ and $b\in (a-\vep,a)$.  Set $\Gamma_0=o$ and define $\Gamma_n$ for $n\ge 1$ recursively as follows: 
\begin{itemize}
\item if   $\Gamma_{n-1}\in \Omega_+$ then we choose our next step to be $e_i$ with probability $\alpha'_i$ independent of the past (here we use the random variable $U_n$), 
\item if $\Gamma_{n-1}\notin \Omega_+$ then we choose our next step to be $e_i$ with probability $(b-p)\alpha'_i/(1-p)$, and $-e_i$ with probability $(1-b)\beta'_i/(1-p)$,  independent of the past, (using $U_n$).  
\end{itemize}
By \eqref{constraint3} if this path ever takes a step $e$ then it a.s.~never takes a step $-e$.  It follows that $\|\Gamma_n\|_1=n$ for every $n$, and that $\Gamma$ is  a self-avoiding path, so the environment seen at every time is a half site with probability $p$ (independent of the past).  It follows that $\Gamma$ is a random walk (i.e.~it has i.i.d.~increments)  with 
\begin{align}
\P(\Gamma_n-\Gamma_{n-1}=e_i)&=p\alpha'_i+(1-p)(b-p)\alpha'_i/(1-p)=b\alpha'_i,\quad \text{ and }\\
\P(\Gamma_n-\Gamma_{n-1}=-e_i)&=(1-p)(1-b)\beta'_i/(1-p)=(1-b)\beta'_i.
\end{align}
The expected increment is then 
\[\mu_b:=\sum_{i=1}^d (b\alpha'_i-(1-b)\beta'_i)e_i.\]
Note that 
\[\mu_b-v=  \sum_{i=1}^d (b-a)(\alpha'_i+\beta'_i)e_i.\]
If for some $j$, $\alpha'_j=\beta'_j=0$ then for every $n$, $\Gamma_n\cdot e_j=0=v\cdot e_j$.  For all other coordinate directions, the law of large numbers for $\Gamma$ and the fact that $(\mu_b-v)\cdot e_i=(b-a)(\alpha'_i+\beta'_i)<0$
 imply that for all $n$ sufficiently large, $n^{-1}\Gamma_n \cdot e_i <v\cdot e_i$.  Moreover, since $a-b<\vep$ we have that $n^{-1}\Gamma_n \cdot e_i>v\cdot e_i-2\vep$ for all $n$ sufficiently large.  Thus for all $n$ sufficiently large, we can reach $m_vnv$ by following the path $\Gamma$ to $\Gamma_{m_vn}$ and then taking at most $2d\vep m_v n$ steps in positive coordinate directions (recall that such steps are possible from every site) to $m_vnv$.  Thus for all $n$ sufficiently large we have that
\[T_{m_vnv}\le T_{\Gamma_{m_vn}}+2d\vep m_v n=m_vn+2d\vep m_v n,\]
which completes the proof.
 \end{proof}

Let $\P^\dagger_q$ denote the law of oriented site percolation on $\Z^d$ with parameter $q$.   
Theorem \ref{thm:features}(c)(ii) is a consequence of the  following ``well-known'' result.
\begin{LEM}
\label{lem:OPhits}
Fix $d\ge 2$ and let $q>p_{\dagger}(d)$.  For $v\in \Z^d$ in the interior of the deterministic cone $O_q$,
\begin{equation}
\P^\dagger_q(o \to nv \text{ i.o.})\ge \inf_{n\in \Z_+}\P^\dagger_q(o \to nv)=\eta(q)>0.
\end{equation}
\end{LEM}
\begin{proof}
The inequality is a standard consequence of continuity of probability measures.  To show that the infimum is positive we will make use of the ``shape theorem'' and ``complete convergence theorem'' for oriented site percolation in general dimensions.  For this purpose it is convenient to express each point $z$ in the positive orthant as $z=(z_{d-1},\|z\|_1)$, where the last coordinate represents the time associated to the point and $z_{d-1}\in \R^{d-1}$. Here $z_{d-1}$ should be interpreted as the projection of $z$ onto the hyperplane $x_1 + \dots + x_d = 0$, i.e.\ the hyperplane orthogonal to the vector $e_1 + \dots  + e_d$ which contains the origin.

Let $v$ be in the interior of $O_q$.  Let $U\subset \R^{d-1}$ denote the asymptotic shape (see e.g.~\cite[Theorem 5]{BG} or \cite[Theorem 1]{IS}) conditional on the cluster of the origin being infinite.  Note that $O_q=\cup_{t\ge 0}(tU)$.  Then $v=(v_{d-1},k)$ where $k=\|v\|_1>0$ and there exists $\vep>0$ such that $v_{d-1}\in (1-\vep)kU$.  
Fix this $\vep$, and let $n \in \N$.  Then $nv_{d-1}\in (1-\vep)nkU$.  Let $t_n=\inf\{j: o \to (nv_{d-1},j)\}$.  
Now observe that 
\begin{align}
\P^\dagger_q(o \to nv)&\ge \sum_{\ell\le nk}\P\big(t_n=\ell, (nv_{d-1},\ell) \to (nv_{d-1},nk)\big)\\
&=\sum_{\ell\le nk}\P(t_n=\ell)\P\big((nv_{d-1},\ell) \to (nv_{d-1},nk)\big)\\
&=\sum_{\ell\le nk}\P(t_n=\ell)\P\big(o \to (o,nk-\ell)\big),
\end{align}
where the sum is over $\ell$ such that $nk-\ell$ is divisible by $d$, and the first equality holds by independence.  
By the complete convergence theorem (e.g.~\cite[Theorem 4]{BG}) we have that $\lim_{n\to \infty}\P^\dagger_q\big(o \to (o,dn)\big)\to c>0$.  Since each point in the positive orthant can be reached with positive probability this implies that $\delta:=\inf_{n\ge 0}\P^\dagger_q\big(o \to (o,dn)\big)>0$.  Thus
\begin{align}
\P^\dagger_q (o \to nv)&\ge \delta \P(t_n \le nk).
\end{align}

By the shape theorem, a.s.~on the event $\{o \to \infty\}$ we have that $t_n\le nk$ for all $n$ sufficiently large.   So for every $\eta>0$ there exists some $n_0$ such that for all $n\ge n_0$, 
\begin{equation}
\P^\dagger_q(o \to nv)\ge \delta \P^\dagger_q(t_n\le nk)\ge \delta (1-\eta)\P^\dagger_q(o \to \infty).
\end{equation}
 Since $\inf_{n\le n_0}\P^\dagger_q(o \to nv) >0$ this completes the proof.
\end{proof}

\begin{proof}[Proof of Theorem \ref{thm:features}(c)(ii)]
	Fix $p$ as in the theorem and such that $1-p>p_c^{osp}(d)$.  Let $v\in \Z^d$ be in the interior of the deterministic cone $O_{1-p}$.  By Lemma \ref{lem:OPhits}, with probability at least $\eta(1-p)>0$, for infinitely many $n$, one can reach $nv$ from $o$ by following a path of full sites using only steps in $O$.  On this event we therefore have that $T_{nv}=\|nv\|_1$ infinitely often.  Since $T_{nv}/n \to \zeta_{p}(v)$ this proves that $v\in \mc{S}_{p}$.   For $u\in \Q^d$ in the interior of $O_{1-p}$ the result follows since $v=m_u u\in \Z^d$.  
\end{proof}

\subsection{Proof of Theorem \ref{thm:features}(d)}
The proof of (d)(i) is standard, but we include it for completeness, and to help the reader understand why the more unusual argument that we present to prove (d)(ii) is used.

\begin{proof}[Proof of Theorem \ref{thm:features}(d)(i)]
Fix $p$ as in the theorem.  
Without loss of generality we may assume that $i=1$.   Let $\vep>0$ be sufficiently small so that 
\begin{equation}
\eta(\vep,\delta):=\Big(\dfrac{(2d-1)(1+\vep)e}{\vep+\delta}\Big)^{\vep+\delta}(1-p)^{1-\delta}<1\label{eps_choice}
\end{equation} 
and $(1-2\delta-\vep)>0$  for any $\delta$ with $|\delta|<\vep$.   This can be done since $1-p<1$ and $(b/x)^x \to 1$ as $x \to 0$ for any $b>0$.

Let $u\in \DD$ with $\|u\|_1=1$ and $u^{[1]}=-(1-\delta)$, where $|\delta|<\vep$.    Suppose that we can reach $un=um_un'$ 
 in at most $(1+\vep)n$ steps.  Then there must be a self-avoiding path $\gamma$ starting at the origin and ending at $un$, of length $\ell(\gamma)\in [n,(1+\vep)n]$,   that is consistent with the environment.   Let $D_n$ be the event that such a path exists.  We will show that $\P(D_n)$ is summable in $n$ and hence by Borel-Cantelli $\P(D_n \text{ i.o.})=0$.  This shows that $\zeta_p(v)\ge(1+\vep)$ and therefore completes the proof.

Let us verify the claim that $\P(D_n)$ is summable.   Let $A_{n,k}$ denote the set of self-avoiding paths $\gamma$ from $o$ to $nu$, with $\ell(\gamma)=k$, and let $B_n=\cup_{k=n}^{\floor{(1+\vep)n}}A_{n,k}$.  
For a path $\gamma\in A_{n,k}$, let $w(\gamma)$ denote the number of steps in direction $-e_1$ taken by the path.  The number of paths of length $k$ with exactly $w$ steps in direction $-e_1$ is at most ${k \choose k-w}(2d-1)^{k-w}$ since there are at most $2d-1$ choices for each of the $k-w$ other steps.   For any path $\gamma\in A_{n,k}$, $w(\gamma)\in [n(1-\delta),k]$.  Now observe that for any $k\in [n,n(1+\vep)]$ and $w\in [n(1-\delta),k]$ we have $0\le k-w\le n(\vep+\delta)$.  

It follows that
\begin{align*}
|B_n|&\le \sum_{k=n}^{\floor{(1+\vep)n}}\sum_{w=\lceil{n(1-\delta)\rceil}}^k{k \choose k-w}(2d-1)^{k-w}\\
&\le \sum_{k=n}^{\floor{(1+\vep)n}}\sum_{w=\lceil{n(1-\delta)\rceil}}^k{k \choose \lceil{n(\vep+\delta)\rceil} }(2d-1)^{n(\vep+\delta)},
\end{align*}
where we have used the fact that for $w,k$ contributing to the sums, $k-w<k/2$ since $(1-2\delta-\vep)>0$.
This is at most
\[\sum_{k=n}^{\floor{(1+\vep)n}}\sum_{w=\lceil{n(1-\delta)\rceil}}^k{\floor{(1+\vep)n} \choose \lceil{n(\vep+\delta)\rceil} }(2d-1)^{n(\vep+\delta)}.\]
Now use the fact that for $k\in 1,\dots, n$, ${n \choose k}\le (ne/k)^k$ (e.g.~proof by induction on $k$) to see that 
\begin{align}
|B_n|&\le \sum_{k=n}^{\floor{(1+\vep)n}}\sum_{w=\lceil{n(1-\delta)\rceil}}^k \Big(\frac{(1+\vep)e}{\vep+\delta}\Big)^{1+n(\vep+\delta)}
(2d-1)^{n(\vep+\delta)}\nn\\
&\le \frac{(1+\vep)e}{\vep+\delta}\Bigg(\Big(\frac{(2d-1)(1+\vep)e}{\vep+\delta}\Big)^{(\vep+\delta)}\Bigg)^n \cdot (1+n\vep) \cdot (1+n(\delta+\vep)).
\end{align}
Finally, note that any $\gamma\in B_n$ is consistent with the environment with probability at most $(1-p)^{w(\gamma)}\le (1-p)^{n(1-\delta)}$.  Therefore 
\begin{align*}
\P(D_n)\le (1-p)^{n(1-\delta)}|B_n| &\leq (\eta(\vep,\delta))^n\frac{(1+\vep)e}{\vep+\delta}(1+n\vep)(1+n(\vep+\delta))\\
&\le 2n^2(\eta(\vep,\delta))^n 
\end{align*}
for $n$ sufficiently large (depending on $\vep$ and $\delta$).
By \eqref{eps_choice}, this is exponentially small in $n$, hence summable, and the claim is proved.
\end{proof}

In the proof of Theorem \ref{thm:features}(d)(i) above, the number of paths grows exponentially in $n$, but the exponential growth can be made as close to 1 as we like by taking $\vep$ small and $u$ very close to $-e_1$ (i.e.~$\delta$ small).  In part (d)(ii) of the theorem we cannot take $u$ close to $-e_1$, so the number of paths grows at an exponential rate that cannot be made close to 1.  For example, the number of self-avoiding paths of length exactly $n$ from $o$ to $(-(1-s)n,sn)$ (where $s\in \Q\in (0,1)$ and $sn\in \Z$) is ${n\choose sn}$ which grows exponentially fast with growth rate depending on $s$.   Thus, in order to prove (d)(ii) of the theorem we need a more sophisticated argument than simple enumeration of paths.  This is the content of the remainder of the paper.  The proof as written works for two dimensions.  It would be of interest to see whether modifications of these arguments can be used to obtain improvements to (d)(i) in higher dimensions.

Henceforth we fix $d=2$.   Let $b(\gamma)$ denote the number of $\rightarrow, \downarrow$ steps of a path $\gamma$.   Given an environment $\bs{\omega}$, a path $\gamma$ in $\Z^2$ starting from $o$ is said to be $\bs{\omega}$-\emph{good} if all $\leftarrow$ steps of the path are consistent with the environment $\bs{\omega}$.  Let $G(\bs{\omega})$ be the set of (finite) $\bo$-good paths.

Given $\bo$ and $\gamma\in G(\bo)$ define the \emph{westernisation} of $\gamma$ to be the path $\tilde{\gamma} \in G(\bo)$ of the same length as $\gamma$ such that the times and types of step in $\{\rightarrow,\downarrow\}$ are identical for the two paths and such that outside of these times, $\tilde{\gamma}$ always takes the $\leftarrow$ step when the environment allows (and $\uparrow$ when it does not).  Recall that the coordinates of a point $x\in \Z^2$ are denoted by $x=(x^{\sss [1]},x^{\sss [2]})$.

\begin{LEM}
\label{lem:western}
Let $\bo$ be an environment and $\gamma\in G(\bo)$.  Then $\gamma$ and its   westernisation $\tilde{\gamma}$ satisfy
 \[\gamma^{\sss [1]}_m-\tilde{\gamma}^{\sss [1]}_m =\gamma^{\sss [2]}_m-\tilde{\gamma}^{\sss [2]}_m\ge 0, \quad  \text{ for all }m \ge 0.\]
\end{LEM}
\begin{proof}
Proof by induction.  The claim is trivially true at time 0.  Suppose the claim is true up to and including time $m-1$.   If at time $m$ one of the paths takes a step in $\{\rightarrow, \downarrow\}$ then the other takes the same step, so the result is also trivially true at time $m$.  

Otherwise, if $\gamma^{\sss [1]}_{m-1}-\tilde{\gamma}^{\sss [1]}_{m-1} =\gamma^{\sss [2]}_{m-1}-\tilde{\gamma}^{\sss [2]}_{m-1}=0$ then both paths are at the same location $x_{m-1}$.  If the local environment $\omega(x_{m-1})$ does not have $\leftarrow$ then both paths take $\uparrow$, and the claim continues to hold (since neither coordinatewise difference changes).  Otherwise if the local environment does have $\leftarrow$ it may be that both paths take this arrow (so the claim continues to hold) or $\gamma$ takes $\uparrow$ and $\tilde{\gamma}$ takes $\leftarrow$ in which case each coordinatewise difference increases by 1, so the claim still holds.

Otherwise if $\gamma^{\sss [1]}_{m-1}-\tilde{\gamma}^{\sss [1]}_{m-1} =\gamma^{\sss [2]}_{m-1}-\tilde{\gamma}^{\sss [2]}_{m-1}>0$, then the only way that either of these increments can change is if one path takes $\leftarrow$ and the other takes $\uparrow$.    But in this case either both differences increase by 1, or both differences decrease by 1, so the claim still holds.  
\end{proof}

Let $\tilde{G}(\bo)\subset G(\bo)$ denote the set of paths in $G(\bo)$ that never take a $\uparrow$ step from a location where $\leftarrow$ is available.

\begin{proof}[Proof of Theorem \ref{thm:features}(d)(ii)]
Let $p$ be as in the statement of the theorem, and $s\in [0,p)\cap \Q$.  Let $\vep\in (0,(p-s)/10)$.  Let $n\in \N$ be such that $ns$ in an integer.

Any path $\gamma$ starting from the origin with length $\ell(\gamma)\le n(1+\vep)$ and taking more than $\vep n$ $\downarrow, \rightarrow$ steps (i.e.~$b(\gamma)\ge \vep n$) cannot reach a point on the boundary of the $\ell_1$ ball $B_1(o,n)$ with first coordinate non-positive and second coordinate  non-negative (i.e.~in the northwest quadrant).  Let 
\begin{align}
J_n=\Big\{&\exists \gamma \text{ consistent with $\bo$ with }\ell=\ell(\gamma)\le n(1+\vep)  \text{ and }b(\gamma)\le \vep n\text{ such that }\\
&0\le \gamma_\ell^{\sss [2]}\le n(s+\vep/4), \gamma_\ell^{\sss [1]}\le -n(1-s-\vep/4),\|\gamma_\ell\|_1\ge n\Big\}.
\end{align}
We will show that there exist $C,c>0$ such that 
\begin{align}
\P(J_n)\le Ce^{-cn}, \quad \text{ for all }n\in \N.
\label{Jnbound}
\end{align}
To see why this is sufficient to prove the result, note that if $\zeta_p(u)=1$ then there must exist infinitely many $n$ such that $n u$ is reached within time $n(1+\vep)$.  But $0\le nu^{\sss [2]} \le n(s+\vep/4)$,  $nu^{\sss [1]}\le -n(1-s-\vep/4)$ and $\|nu\|_1=n$, so this means that $J_{n}$ occurs.  By \eqref{Jnbound} and Borel-Cantelli this can only occur for finitely many $n$.  Thus the subsequence of $n^{-1}T_{[nu]}$ for which $nu\in \Z^2$ a.s. does not converge to 1.  Thus we must have $\zeta_p(u)>1$.

The left hand side of \eqref{Jnbound} is at most 
\begin{align}
\P\Big(&\exists \gamma\in G(\bo) \text{ with }\ell(\gamma)\le n(1+\vep)  \text{ and }b(\gamma)\le \vep n\text{ such that }\nn\\
&0\le \gamma_\ell^{\sss [2]}\le n(s+\vep/4), \gamma_\ell^{\sss [1]}\le -n(1-s-\vep/4),\|\gamma_\ell\|_1\ge n\Big).\label{pp1}
\end{align}
By Lemma \ref{lem:western} if there exists a $\gamma$ as in this event then its westernisation $\tilde{\gamma}$ also satisfies the constraints in \eqref{pp1} except that we may have $\tilde{\gamma}_\ell^{\sss [2]}<0$.   In particular we have that for such a $\gamma$,
\begin{align*}
\|\tilde{\gamma}_\ell\|_1&=-\tilde{\gamma}_\ell^{\sss[1]}+|\tilde{\gamma}_\ell^{\sss[2]}|\\
		&=-\gamma_\ell^{\sss[1]}+( \gamma_\ell^{\sss[1]}  - \tilde{\gamma}_\ell^{\sss[1]})
+			|\tilde{\gamma}_\ell^{\sss[2]}|\\
				&=-\gamma_\ell^{\sss[1]}+( \gamma_\ell^{\sss[2]}   -\tilde{\gamma}_\ell^{\sss[2]})+							|\tilde{\gamma}_\ell^{\sss[2]}|\\
								&\ge -\gamma_\ell^{\sss[1]}+\gamma_\ell^{\sss[2]}=\|\gamma_\ell\|_1\ge n.
\end{align*}

 Thus, \eqref{pp1} is at most
\begin{align}
\P\Big(&\exists \gamma\in \tilde{G}(\bo) \text{ with }\ell(\gamma)\le n(1+\vep)  \text{ and }b(\gamma)\le \vep n\text{ such that }\nn\\
&\gamma_\ell^{\sss [2]}\le n(s+\vep/4), \gamma_\ell^{\sss [1]}\le -n(1-s-\vep/4),\|\gamma_\ell\|_1\ge n\Big).\label{pp2}
\end{align}
A path $\gamma$ such as that in \eqref{pp2} need not be self-avoiding.  If it is not self-avoiding then by removing loops, it contains a shorter path $\gamma'$ with $\ell(\gamma')\le \ell(\gamma)$ and $b(\gamma')\le b(\gamma)$ that ends at the same point as $\gamma$.   Let $\tilde{S}(\bo)\subset \tilde{G}(\bo) $ denote the set of paths in $\tilde{G}(\bo) $ that are self-avoiding.

Then \eqref{pp2} is at  most
\begin{align}
&\P\Big(\exists \gamma\in \tilde{S}(\bo) \text{ with }\ell(\gamma)\le n(1+\vep)  \text{ and }b(\gamma)\le \vep n\text{ such that }\nn\\
& \qquad \gamma_\ell^{\sss [2]}\le n(s+\vep/4), \gamma_\ell^{\sss [1]}\le -n(1-s-\vep/4),\|\gamma_\ell\|_1\ge n\Big)\nn\\
&\le \sum_{l=n}^{n(1+\vep)}\sum_{b=0}^{\vep n}\P\Big(\exists \gamma\in \tilde{S}(\bo) \text{ with }\ell(\gamma)=l  \text{ and }b(\gamma)=b\text{ such that }\label{pp3}\\
&\phantom{\le \sum_{l=n}^{n(1+\vep)}\sum_{b=1}^{\vep n}\P\Big(\exists }
 \gamma_l^{\sss [2]}\le n(s+\vep/4), \gamma_l^{\sss [1]}\le -n(1-s-\vep/4),\|\gamma_l\|_1\ge n\Big).\nn
\end{align}
The upper limits of these sums are $\floor{n(1+\vep)}$ and $\floor{n\vep}$ respectively since $l$ and $b$ are integer-valued.
 Now we sum over the times $\vec{t}(\gamma)=\{t_1(\gamma), \dots, t_b(\gamma)\}$ at which the path $\gamma$ takes a $\rightarrow, \downarrow$ step and a sum over the exact sequence $\vec{x}(\gamma)\in \{\rightarrow,\downarrow\}^b$ of such  steps to see that \eqref{pp3} is at most
 \begin{align}
 \sum_{l=n}^{n(1+\vep)}\sum_{b=0}^{\vep n}\sum_{\vec{t}}\sum_{\vec{x}}
 \P\Big(&\exists \gamma\in \tilde{S}(\bo) \text{ with }\ell(\gamma)=l,\vec{t}(\gamma)=\vec{t}  \text{ and }\vec{x}(\gamma)=\vec{x} \text{ such that }\nn\\
&
 \gamma_l^{\sss [2]}\le n(s+\vep/4), \gamma_l^{\sss [1]}\le -n(1-s-\vep/4),\|\gamma_l\|_1\ge n\Big).\label{pp4}
\end{align}
Now consider how to determine whether the above event occurs.  Construct a random path $\Gamma:=\Gamma(\bo,\vec{t},\vec{x})$ of length $l$ as follows.  Starting from the origin at time $0$ (set $\Gamma_0=o$), if $j =t_i \in \vec{t}$ then let $\Gamma$ take step $x_i$ (i.e.~$\Gamma_{j+1}=\Gamma_{j}+x_i$).  If not, look at the environment $\omega_{\Gamma_{j}}$: if this contains $\leftarrow$ then $\Gamma_{j+1}=\Gamma_{j}-e_1$, and otherwise $\Gamma_{j+1}=\Gamma_{j}+e_2$.  If this path meets itself at some time then it is not self-avoiding, so the event has not occurred.  As long as this path never meets itself then (each new environment that is looked at has the usual law, is independent of the past and) we can ask whether  $\Gamma_l^{\sss [2]}\le n(s+\vep/4)$ and $\Gamma_l^{\sss [1]}\le -n(1-s-\vep/4)$, and $\|\Gamma_l\|_1\ge n$.  If all of these are true then the event has occurred.  Otherwise the event has not occurred.  It follows that the event in \eqref{pp4} is equal to 
\begin{align}
\Big\{\Gamma \text{ is self-avoiding, with }\Gamma_l^{\sss [2]}\le n(s+\vep/4), \Gamma_l^{\sss [1]}\le -n(1-s-\vep/4),\|\Gamma_l\|_1\ge n\Big\}.
\end{align}

Now, let $\bs{X}:=(X_i)_{i\ge 0}$ be an i.i.d.~sequence with $\P(X_i=\leftarrow)=1-p=1-\P(X_i =\uparrow)$, and fix $b$ and $\vec{t}$ and $\vec{x}$.  Let $\bar{\Gamma}=\bar{\Gamma}(\bs{X},\vec{t},\vec{x})$ be a path of length $l$ in $\Z^2$ defined as follows:
\begin{itemize}
\item $\bar{\Gamma}_0=o$
\item if $j=t_i\in \vec{t}$ then $\bar{\Gamma}_{j+1}-\bar{\Gamma}_{j}$ is equal to $x_i$
\item otherwise $\bar{\Gamma}_{j+1}-\bar{\Gamma}_{j}$ is equal to $X_j$.
\end{itemize}
The probability in \eqref{pp4} is equal to 
\begin{align}
&\P\Big(\bar{\Gamma} \text{ is self-avoiding, } \bar\Gamma_l^{\sss [2]}\le n(s+\vep/4), \Bar\Gamma_l^{\sss [1]}\le -n(1-s-\vep/4),\|\Bar\Gamma_l\|_1\ge n\Big)\nn\\
&\le\P\Big( \bar\Gamma_l^{\sss [2]}\le n(s+\vep/4), \Bar\Gamma_l^{\sss [1]}\le -n(1-s-\vep/4),\|\Bar\Gamma_l\|_1\ge n\Big)\nn\\
&\le \P\Big( \bar\Gamma_l^{\sss [2]}\le n(s+\vep/4)\Big).\label{hippo1}
\end{align}
Now note that $\bar{\Gamma}$ is a (biased) simple random walk of length $l-b$, with $b$ deterministic $\downarrow, \rightarrow$ steps inserted at specific times in the path.  Next note that $b$ is very small compared to $l$, and we are asking for this walk to deviate far from its mean at time $l$ of order $n$.  We now show that this probability is exponentially small with exponent that can be taken uniform in small $\vep$.  To be precise, we will show that there exist $c,C>0$ depending on $p,s$ (but not $\vep$) such that for all $l\in [n,n(1+\vep)]\cap \Z_+$, $b\le n\vep$, $\vec{t}$, and $\vec{x}$, and all $\vep>0$ sufficiently small,
\begin{equation}
\P\Big( \bar\Gamma_l^{\sss [2]}\le n(s+\vep/4)\Big)\le Ce^{-cn}.\label{hippo2}
\end{equation}
Let us assume that \eqref{hippo2} holds and explain how to proceed.  

Using \eqref{hippo2} we have that \eqref{pp4} is at most
\begin{align}
\sum_{l=n}^{n(1+\vep)}\sum_{b=0}^{\vep n}\sum_{\vec{t}}\sum_{\vec{x}}Ce^{-cn}.\label{hippo3}
\end{align}
The sum over $\vec{x}$ gives a contribution at most $2^{b}\le 2^{\vep n}$.  The sum over $\vec{t}$ contributes at most ${\floor{n(1+\vep)} \choose \lceil{n\vep\rceil}}$ when $n$ is large and $\vep<1/3$.  This combinatorial factor is at most
\begin{align}
\dfrac{C(n+n\vep)^{n+n\vep+1/2}}{
(n\vep)^{n\vep+1/2}(n-1)^{n-1+1/2}}\le \dfrac{C(n+n\vep)^{n+n\vep}}{
(n\vep)^{n\vep}(n-1)^{n-1}},
\end{align}
for all $n$ sufficiently large depending on $\vep$.  This is at most
\begin{align}
 \dfrac{C'n^{n+n\vep}(1+\vep)^{n+n\vep}}
 {n^{n\vep}\vep^{n\vep}n^{n-1}}= \dfrac{C'n(1+\vep)^{n+n\vep}}
 {\vep^{n\vep}}=C'n \Bigg(\dfrac{(1+\vep)^{1+\vep}}{\vep^\vep}\Bigg)^n.
\end{align}
The exponential growth rate of this quantity is 
\begin{align}
(1+\vep) (1+1/\vep)^{\vep}=(1+\vep)e^{\vep \log(1+1/\vep)},
\end{align}
which approaches $1\times e^0=1$ as $\vep \downarrow 0$.  Thus for small $\vep>0$, for all $n$ sufficiently large \eqref{hippo3} is at most
\begin{equation}
\sum_{l=n}^{n(1+\vep)}\sum_{b=0}^{\vep n}Cn(1+a_\vep)^n 2^{n\vep} Ce^{-cn}\le C''n^3e^{-c'n}\le C'''e^{-c''n},
\end{equation}
We have shown that \eqref{Jnbound} holds (and hence the theorem is proved), assuming \eqref{hippo2}.  It therefore remains to prove that for some $c,C>0$ (not depending on $\vep$) \eqref{hippo2} holds for all $l,b,\vec{t},\vec{x}$.  Fix $l,b,\vec{t},\vec{x}$ and let $N_{r}=\sum_{i=1}^{r}\indic{X_i=\uparrow}$ be the number of $\uparrow$ steps among the first $r$ of the $X$'s.  Then \eqref{hippo1} is at most 
\begin{align}
\P\Big(N_{l-b}\le n(s+\vep/4)+n\vep\Big)&=\P\Big(N_{l-b}\le p(l-b)-(p(l-b)-n(s+5\vep/4))\Big).\label{dog1}
\end{align}
Here, $N_{l-b}$ is a sum of $l-b$ i.i.d.~(0-1-valued) random variables each with expectation $p$, so $\E[N_{l-b}]=p(l-b)$.  Since $b\le \vep n$ and $l\ge n$ we have 
\begin{align}
p(l-b)-n(s+5\vep/4)&\ge pn(1-\vep)-n(s+5\vep/4)\\
&=n(p(1-\vep)-s-5\vep/4).
\end{align}
Since $\vep<(p-s)/10$ we have
\begin{align}
p(1-\vep)-s-5\vep/4&>p(1-(p-s)/10)-s-5(p-s)/10\\
&=(p-s)[1-(p+5)/10]\\
&>4(p-s)/10.
\end{align}
This means that 
\[p(l-b)-n(s+5\vep/4)\ge 4n(p-s)/10\ge 2(l-b)(p-s)/10,\]
where we have used the fact that $l-b\le l\le n(1+\vep)\le 2n$.

Thus \eqref{dog1} is at most
\begin{align}
\P\Big(N_{l-b}\le (l-b)[p-2(p-s)/10]\Big)\le Ce^{-c(l-b)},
\end{align}
for some $C,c$ depending only on $p,s$ (e.g.~by Hoeffding's inequality).  Since $l-b\ge n(1-\vep)>n/2$ (for all $\vep$ sufficiently small), this verifies \eqref{hippo2} as claimed.
\end{proof}

\section*{Acknowledgements}
Part of this research was undertaken while MH was supported by a Future Fellowship FT160100166 from the Australian Research Council.  We thank Geoffrey Grimmett for helpful discussions regarding Lemma \ref{lem:OPhits}, and for providing the proof of Lemma \ref{lem:siteperc}.  MH thanks Tom Salisbury for spotting an error in an earlier version of the document.

\bibliographystyle{plain}

\end{document}